\documentclass[12pt,letterpaper]{amsart}
\usepackage{palatino, euler, epic, eepic, amssymb, xypic, floatflt, microtype, enumerate, stmaryrd}
\input xy
\xyoption{all}

 \newlength{\baseunit}               
 \newcount{\numlines}                
\usepackage{amsmath, amssymb, amscd, verbatim, xspace,amsthm}
\usepackage{latexsym, epsfig, color}

\newcommand{\ra}{\rightarrow}

\newcommand\Hom{\operatorname{Hom}}

\newcommand\Sym{\operatorname{Sym}}

\newcommand\isom{\cong}

\newcommand\Spec{\operatorname{Spec}}
\newcommand\Proj{\operatorname{Proj}}

\newcommand\bq{\begin{equation}}
\newcommand\eq{\end{equation}}

\newtheorem{proposition}{Proposition}[section]
\newtheorem{theorem}[proposition]{Theorem}
\newtheorem{corollary}[proposition]{Corollary}

\newtheorem{lemma}[proposition]{Lemma}

\newtheorem*{claim}{Claim}

\theoremstyle{definition}
\newtheorem{definition}[proposition]{Definition}
\newtheorem{notation}[proposition]{Notation}

\theoremstyle{remark}
\newtheorem{remark}[proposition]{Remark}
\usepackage{url}

\numberwithin{equation}{section}

\newcommand{\cut}[1]{}

\newcommand\hidden[1]{}

\newcommand{\cC}{{\mathcal{C}}}

\newcommand{\cM}{\mathcal{M}}

\newcommand{\PP}{\mathbb{P}}

\newcommand{\RR}{\mathbb{R}}

\newcommand{\isomto}{\stackrel{\isom}{\to}}                           %
\newcommand{\setmin}{{\smallsetminus}}                                %
\newcommand{\Mp}{{\mathcal{M}^+}}                                     %
\newcommand{\Rel}{{\mathcal{R}}}                                      %
\newcommand{\lar}{\longleftarrow}                                     %
\newcommand{\Span}{\operatorname{Span}}                               %
\newcommand{\ZZ}{{\mathbb{Z}}}                                        %
\newcommand{\LL}{{\mathbb{L}}}                                        %
\newcommand{\cO}{{\mathcal O}}                                        %
\newcommand{\ch}{\operatorname{\tilde{c}_1}}                          %
\newcommand{\St}{\operatorname{St}}                                   %
\newcommand{\res}{\operatorname{res}}                                 %
\newcommand{\Po}{Poincar\'e }                                         %
\newcommand{\Schk}{\operatorname{Sch}_{k}}                            %
\newcommand{\Schkg}{{G}{\textnormal{-}}{\operatorname{Sch}_{k}}}      %
\newcommand{\Smk}{\operatorname{Sm}_{k}}                              %
\newcommand{\Ab}{\operatorname{Ab}_{*}}                               %
\newcommand{\id}{\operatorname{Id}}                                   %
\newcommand{\fxy}{f:X \ra Y}                                          %
\newcommand{\fyx}{f:Y \ra X}                                          %
\newcommand{\fxz}{f:X \ra Z}                                          %
\newcommand{\fzx}{f:Z \ra X}                                          %
\newcommand{\gyx}{g:Y \ra X}                                          %
\newcommand{\gyz}{g:Y \ra Z}                                          %
\title{Bivariant Algebraic Cobordism}

\author{Jos\'e Luis Gonz\'alez and Kalle Karu}
\address{J.L. Gonz\'alez,  Dept. of Mathematics, University of British Columbia,
  Vancouver, BC V6T1Z2, CANADA  \newline \indent
K. Karu,
Dept. of Mathematics, University of British Columbia, 
  Vancouver, BC V6T1Z2, CANADA} 
\email{jgonza@math.ubc.ca, karu@math.ubc.ca}
\thanks{This research was funded by NSERC Discovery and Accelerator grants.}

\begin{document}

\begin{abstract}
We associate a bivariant theory to any suitable oriented Borel-Moore homology theory on the category of algebraic schemes or the category of algebraic $G$-schemes.
%
%
Applying this to the theory of algebraic cobordism yields operational cobordism rings and operational $G$-equivariant cobordism rings associated to all schemes in these categories.  
In the case of toric varieties, the operational $T$-equivariant cobordism ring may be described as the ring of piecewise graded power series on the fan with coefficients in the Lazard ring.
\end{abstract}
\maketitle
\setcounter{tocdepth}{1} 
\tableofcontents



\section{Introduction}

The purpose of this article is to study the \emph{operational bivariant theory} $B^*$ associated to a \emph{refined oriented Borel-Moore pre-homology theory} $B_*$, and the equivariant versions of these theories. We apply this to the algebraic cobordism theory $\Omega_*$ of Levine and Morel \cite{Levine-Morel} to construct the \emph{operational bivariant cobordism theory} $\Omega^*$. As an application, we describe the \emph{operational} $T$-\emph{equivariant cobordism} $\Omega^*_T(X_\Delta)$ for a toric variety $X_\Delta$.

Bivariant theories were defined by Fulton and MacPherson \cite{Fulton-MacPherson, FultonIT}. A bivariant theory assigns a group $B^*(X\to Y)$ to every morphism $X\to Y$ of schemes. The theory contains both a covariant homology theory $B_*(X) = B^*(X\to pt)$ and a contravariant cohomology theory $B^*(X) = B^*(\id_X: X\to X)$, but the bivariant theory can be more general in the sense that there may be invariants of the theory that are not determined by homology and cohomology alone \cite{Fulton-MacPherson}. Starting with a suitable homology theory $B_*(X)$, one can extend it to a bivariant theory $B^*$. The definition of this bivariant theory is a generalization of the construction of Fulton and MacPherson in \cite{Fulton-MacPherson, FultonIT} of a bivariant theory $A^{*}$ associated to the Chow theory $A_{*}$. The elements of  $B^*(X\to Y)$ are certain compatible operators on the homology groups $B_*$, hence the bivariant theory is called the operational bivariant theory. The only bivariant theories we consider in this paper are the operational ones. An operational bivariant theory can be 
viewed as a method of constructing a cohomology theory $B^*$ out of a homology theory $B_*$. The cohomology theory takes values in rings, hence there is a well-defined intersection product in this theory. The Chern class operators naturally lie in the cohomology $B^*$.

The operational cohomology theory $B^*$ at first seems very intractable. A single element of $B^*(X)$ is defined by an infinite set of homomorphisms. However, Kimura \cite{Kimura} has shown that, in case of Chow theory $A_*$, the bivariant cohomology groups $A^*(X)$ for an arbitrary variety $X$ can often be computed  if one knows the homology groups $A_*(Y)$ for smooth varieties $Y$. Payne \cite{Payne} carried out this computation for the equivariant Chow cohomology $A^*_T$ of toric varieties. By a result of Brion \cite{Brion}, the $T$-equivariant Chow ring $A^*_T(X_\Delta)$ of a smooth toric variety $X_\Delta$ can be identified with the group of integral piecewise polynomial functions on the fan $\Delta$. Payne showed that the ring of such functions on an arbitrary fan $\Delta$ gives the operational $T$-equivariant Chow cohomology $A^*_T(X_\Delta)$. A similar computation in the case of $K$-theory is done by Anderson and Payne \cite{Anderson-Payne}. Brion and Vergne \cite{Brion-Vergne} (see also Vezzosi and Vistoli \cite{Vezzosi-Vistoli}) have proved that the $T$-
equivariant $K$-theory ring of a smooth toric variety $X_\Delta$ is isomorphic to the ring of integral piecewise exponential functions on the fan $\Delta$. Anderson and Payne show that for an arbitrary fan $\Delta$ this ring gives the operational $T$-equivariant $K$-cohomology of the variety $X_\Delta$. 

One of the goals of this article is to extend the results of Payne and Anderson-Payne to the case of algebraic cobordism. The $T$-equivariant algebraic cobordism of smooth toric varieties was computed by Krishna and Uma \cite{Krishna-Uma}. Using the same terminology as in the case of Chow theory and $K$-theory, the equivariant cobordism ring $\Omega^*_T(X_\Delta)$ of a smooth toric variety $X_\Delta$ can be identified with the ring of piecewise graded power series on the fan $\Delta$, with coefficients in the Lazard ring $\LL$. We will prove that for a (quasiprojective) fan $\Delta$ the same ring is isomorphic to the operational $T$-equivariant cobordism ring $\Omega^*_T(X_\Delta)$ of the variety $X_\Delta$.

We start by constructing the operational bivariant theory $B^*$ for as large a class of homology theories $B_*$ as possible. To carry out the construction of the operational bivariant theory, it suffices to assume that $B_*$ is a \emph{refined oriented Borel-Moore pre-homology theory} (ROBM pre-homology theory). This is a weakening of the notion of oriented Borel-Moore homology theory \cite[Definition 5.3.1]{Levine-Morel} with refined Gysin homomorphisms, where we do not require the projective bundle, extended homotopy and the cellular decomposition properties. The various constructions can be summarized by a diagram as follows:
 \[ \begin{CD}  
 B_*^G @>>> B^*_G  \\
@AAA @AAA\\
 B_* @>>>  B^*.\\
\end{CD} \]
Each horizontal arrow associates to an ROBM pre-homology theory its operational bivariant theory.
This step can be applied to an arbitrary ROBM pre-homology theory $B_*$, including its equivariant version $B_*^G$ for a linear algebraic group $G$. The vertical arrows associate to a theory its $G$-equivariant version using Totaro's algebraic approximation of the Borel construction from topology \cite{Totaro}. For these constructions to be well-defined, we need to assume
that the ROBM pre-homology theory $B_*$ has the localization and homotopy properties. The construction of $B_*^G$  is a direct generalization of similar constructions in Chow theory by Totaro \cite{Totaro} and  by  Edidin and Graham \cite{Edidin-Graham}, and in algebraic cobordism by Krishna \cite{Krishna} and by Heller and Malag\'on-L\'opez \cite{Heller-Malagon}. 

We will prove that the above square commutes; more precisely, the two ways to construct $B^*_G$ agree if we assume that the original theory $B_*$ has certain exact descent sequences for envelopes. Such sequences were first proved by Gillet \cite{GilletSeq} in $K$-theory and Chow theory, and they were used by Edidin and Graham \cite{Edidin-Graham} to prove the commutativity of the square above for Chow theory. The Edidin-Graham proof can be generalized to an arbitrary ROBM pre-homology theory, but the descent property depends on the theory. The descent property for the algebraic cobordism theory was proved in \cite{descentseq}. 

The descent property in the Chow theory was used by Kimura in \cite{Kimura} to give an inductive construction of operational Chow cohomology classes. We will generalize Kimura's proofs to arbitrary ROBM pre-homology theories that satisfy the descent property.

Levine and Morel \cite{Levine-Morel} showed that the algebraic cobordism theory is universal among all oriented Borel-Moore homology theories. We can not prove any similar universality statement for the operational bivariant theory or its cohomology. Yokura \cite{Yokura} has proposed a geometric method for constructing a bivariant algebraic cobordism theory $\tilde{\Omega}^*$, which would be universal among a class of  oriented bivariant theories. The homology of this bivariant theory $\tilde{\Omega}^*$ is expected to be the algebraic cobordism $\Omega_*$.  By universality, there should exist a natural transformation from Yokura's bivariant $\tilde{\Omega}^*$ to the operational $\Omega^*$, restricting to an isomorphism between the homology theories.  To relate these two theories would then be an interesting problem.

The paper is organized as follows. We define refined oriented Borel-Moore pre-homology theories in \S\,\ref{section.definition.obm}. In \S\,\ref{section.bivariant.theories} we define bivariant theories and associate the operational bivariant theory $B^{*}$ to any ROBM pre-homology theory $B_{*}$ in the categories $\Schk$ and $\Schkg$, which among other properties has the original theory $B_{*}$ as its associated homology theory (see Proposition~\ref{proposition.homology}) and has the Poincar\'e duality isomorphism between homology and cohomology in the nonsingular case (see Proposition~\ref{poincare}). In \S\,\ref{section.equivariant.theories} we start from any ROBM pre-homology theory $B_{*}$ on $\Schk$ that satisfies the localization and homotopy properties and construct the $G$-equivariant ROBM pre-homology theory $B^{G}_{*}$ on $\Schkg$ by taking a limit over successively better approximations of the Borel construction. In \S\,\ref{section.kimura.bivariant} we show that if $B_{*}$ has exact descent sequences (\ref{sequence.envelopes}), then the computation of bivariant classes can be inductively reduced to the nonsingular case (see Theorem~\ref{thm-biv-image} and Theorem~\ref{thm-seq2-biv}), and that furthermore the operational equivariant theory $B^{*}_{G}$ can alternatively be computed by applying the limit construction directly to the operational theory $B^{*}$ (see Proposition~\ref{proposition.limit}).

In \S\,\ref{section.overview.cobordism} we overview the theory of algebraic cobordism $\Omega_{*}$. We conclude this article in \S\,\ref{section.cobordism.toric} by showing in Theorem~\ref{thm.toric} that the operational $T$-equivariant cobordism ring of a toric variety can be described as the ring of piecewise graded power series on the fan with coefficients in the Lazard ring.


\subsection*{Acknowledgments}
The authors would like to thank Dave Anderson, William Fulton and Sam Payne for insightful conversations. This work started by trying to understand the operational equivariant K-theory of toric varieties constructed by Dave Anderson and Sam Payne.




\section{Refined Oriented Borel-Moore pre-Homology Theories}   \label{section.definition.obm}

\subsection{Notation and Conventions}

\subsubsection{} \label{setting}

Throughout this article all of our schemes will be defined over a fixed field $k$. We denote by $\Schk$ the category of separated finite type schemes over $\Spec k$ and by $\Schk'$ the subcategory of $\Schk$ with the same objects but whose morphisms are the projective morphisms. We denote by $\Smk$ the full subcategory of $\Schk$ of smooth and quasiprojective schemes. By a smooth morphism we always mean a smooth and quasiprojective morphism.
$\Ab$ will denote the category of graded abelian groups.

\subsubsection{} \label{equivariant.setting}

Let $G$ be a linear algebraic group. A $G$-linearization of a line bundle $f:L \ra X$ over the $G$-scheme $X$, is a $G$-action $\Phi:G \times L \ra L$ on $L$ such that $f$ is $G$-equivariant and for every $x \in X$ and $g \in G$ the action map $\Phi_{g}:L_{x} \ra L_{gx}$ is linear. We denote by $\Schkg$ the category whose objects are the separated finite type $G$-schemes over $\Spec k$ that admit an ample $G$-linearizable line bundle and whose morphisms are $G$-equivariant morphisms. We denote by ${G}{\textnormal{-}}{\operatorname{Sch}_{k}'}$ the subcategory of $\Schkg$ with the same objects but whose morphisms are the projective $G$-equivariant morphisms. Note that all schemes in $\Schkg$ are assumed to be quasiprojective; this is needed in the construction of equivariant theories using the GIT quotients.

\subsubsection{} \label{assumptions.field}

In \S\S\,\ref{section.kimura.bivariant}-\ref{section.cobordism.toric} we will assume that $k$ has characteristic zero and in \S\S\,\ref{construction.induced.equivariant.OBM}-\ref{question.oe.eo} 
we will assume that $k$ is infinite. The assumption on the characteristic of $k$ is only meant to guarantee the existence of smooth projective envelopes in the categories $\Schk$ and $\Schkg$ and to provide the setting for the use of Levine-Pandharipande's version of algebraic cobordism, which requires resolution of singularities by projective morphisms, weak factorization for birational maps and some Bertini-type theorems that hold in characteristic zero. The assumption on the cardinality of $k$ is only used explicitly in the proof of Proposition~\ref{proposition.homotopy.refined}.

\subsubsection{}
We call a morphism $\fzx$ in one of the categories $\mathcal{C}=\Schk$ or $\mathcal{C}=\Schkg$ a locally complete intersection morphism in $\mathcal{C}$ or simply an l.c.i. morphism in $\mathcal{C}$, if there exist a regular embedding $i:Z \ra Y$ and a smooth morphism $g: Y \ra X$, with $g$ and $i$ in $\mathcal{C}$ such that $f=gi$. When we work in the category $\Schkg$ and we say that a morphism $f$ is an equivariant l.c.i. morphism or simply an l.c.i. morphism we mean that $f$ is an l.c.i. morphism in $\Schkg$. We follow the convention that smooth morphisms, and more generally l.c.i. morphisms, are assumed to have a relative dimension. If $\fxy$ is an l.c.i. morphism of relative dimension $d$ (or relative codimension $-d$) and $Y$ is irreducible, then $X$ is a scheme of pure dimension equal to $\operatorname{dim}Y+d$.


\subsection{ROBM pre-Homology Theories}

\subsubsection{} \label{conventions}
For simplicity, we unify the treatment of the cases when the ambient category is $\Schk$ or $\Schkg$ for some algebraic group $G$. Therefore, through the rest of this section we fix the category $\mathcal{C}$, which is either $\Schk$ or $\Schkg$, and we assume that all the schemes and morphisms are in $\mathcal{C}$ (e.g. the statement \emph{for any morphism} should be interpreted as \emph{for any morphism in $\mathcal{C}$}). Likewise, when $\mathcal{C}=\Schkg$, by an l.c.i. morphism we mean an equivariant l.c.i. morphism. The category $\mathcal{C}'$ is defined to be $\Schk'$ or ${G}{\textnormal{-}}{\operatorname{Sch}_{k}'}$, depending on whether $\mathcal{C}$ is equal to $\Schk$ or $\Schkg$, respectively.

Let us start by recalling the definition of a Borel-Moore functor on $\cC$ and several extra structures on it from \cite{Levine-Morel}.

\begin{definition}
A \emph{Borel-Moore functor} on $\cC$ is given by:

$\operatorname{(D_{1})}$ An \emph{additive} functor $H_{*}: \cC' \ra \Ab$, i.e., a functor $H_{*}: \cC' \ra \Ab$ such that for any finite family $(X_{1},\ldots,X_{r})$ of schemes in $\cC'$, the morphism
	\[
\bigoplus_{i=1}^{r}H_{*}(X_{i}) \ra H_{*}(\coprod_{i=1}^{r}X_{i})
\]
induced by the projective morphisms $X_{i} \subseteq \coprod_{i=1}^{r}X_{i}$ is an isomorphism.

$\operatorname{(D_{2})}$ For each smooth equidimensional morphism $\fyx$ of relative dimension $d$ in $\cC$ a homomorphism of graded groups
	\[
f^{*}:H_{*}(X)\ra H_{*+d}(Y).	
\]

These data satisfy the following axioms:

$\operatorname{(A_{1})}$ For any pair of composable smooth equidimensional morphisms $(\fyx, g:Z \ra Y)$ of relative dimensions $d$ and $e$ respectively, one has 
	\[
(f \circ g)^{*} = g^{*} \circ f^{*}: H_{*}(X)\ra H_{*+d+e}(Z).	
\]
In addition, $\id_{X}^{*}=\id_{H_{*}(X)}$ for any $X \in \cC$.

$\operatorname{(A_{2})}$ For any projective morphism $\fxz$ and any smooth equidimensional morphism $\gyz$, if one forms the fiber diagram
\[
\xymatrix{
W \ar[d]^{f'} \ar[r]^{g'} &X \ar[d]^f \\
Y \ar[r]^g&Z }
\]
then
	\[
g^{*}f_{*}=f'_{*}g'^{*}.
\]

\end{definition}

\begin{notation} For each projective morphism $f$ the homomorphism $H_{*}(f)$ is denoted $f_{*}$ and called the \emph{push-forward along} $f$. For each smooth equidimensional morphism $g$ the homomorphism $g^{*}$ is called the \emph{pull-back along} $g$.
\end{notation}

%
%

\begin{definition} \label{def-ext-prod}
A \emph{Borel-Moore functor with exterior product} on $\cC$ consists of a Borel-Moore functor $H_{*}$ on $\cC$, together with:

$\operatorname{(D_{3})}$ An element $1 \in H_{0}(\Spec k)$ and for each pair $(X,Y)$ of schemes in $\cC$, a bilinear graded pairing (called the \emph{exterior product})
	\begin{align*}
\times: H_{*}(X) \times H_{*}(Y) &\ra H_{*}(X \times Y)   \\
(\alpha, \beta)	&\mapsto \alpha \times \beta
\end{align*}
which is (strictly) commutative, associative, and admits $1$ as unit.

These satisfy

$\operatorname{(A_{3})}$ Given projective morphisms $f:X\ra X'$ and $g:Y\ra Y'$ one has that for any classes $\alpha \in H_{*}(X)$ and $\beta \in H_{*}(Y)$ 
	\[
	(f \times g)_{*}(\alpha \times \beta)   =    f_{*}(\alpha) \times g_{*}(\beta)   \in H_{*} (X' \times Y').
\]

$\operatorname{(A_{4})}$ Given smooth equidimensional morphisms $f:X\ra X'$ and $g:Y\ra Y'$ one has that for any classes $\alpha \in H_{*}(X')$ and $\beta \in H_{*}(Y')$
	\[
	(f \times g)^{*}(\alpha \times \beta)   =    f^{*}(\alpha) \times g^{*}(\beta)  \in H_{*} (X \times Y).
\]

\end{definition}

\begin{remark}
Given a Borel-Moore functor with exterior product $H_{*}$, the axioms give $H_{*}(\Spec k)$ a commutative, graded ring structure, give to each $H_{*}(X)$ the structure of $H_{*}(\Spec k)$-module, and imply that the operations $f_{*}$ and $f^{*}$
preserve the $H_{*}(\Spec k)$-module structure. 
\end{remark}

%
%

\begin{definition}
A \emph{Borel-Moore functor with intersection products} on $\cC$ 
is a Borel-Moore functor $H_{*}$ on $\cC$, together with:

$\operatorname{(D_{4})}$ For each l.c.i. morphism $\fzx$ of relative codimension $d$ and any morphism $\gyx$ giving the fiber diagram 
\[
\xymatrix{
W \ar[d]^{g'} \ar[r]^{f'} &Y \ar[d]^g \\
Z \ar[r]^f&X, }
\]
a homomorphism of graded groups
	\[
f^{!}_{g}: H_{*}(Y)\ra H_{*-d}(W).	
\]

These satisfy 

$\operatorname{(A_{5})}$ If $f_{1}:Z_{1}\ra X$ and $f_{2}:Z_{2}\ra Z_{1}$ are l.c.i. morphisms and $g:Y\to X$ any morphism giving the fiber diagram
\[
\xymatrix{
W_{2} \ar[d]^{} \ar[r]^{f'_{2}} &W_{1} \ar[d]^{g'} \ar[r]^{f'_{1}} &Y\ar[d]^{g}\\
Z_{2} \ar[r]^{f_{2}} &Z_{1} \ar[r]^{f_{1}} &X,}
\]
one has $( f_{1} \circ f_{2} )^{!}_{g}  = {(f_{2})}^{!}_{g'}  \circ {(f_{1})}^{!}_{g}$.

$\operatorname{(A_{6})}$ If $f_{1}:Z_{1} \ra X_{1}$ and $f_{2}: Z_{2} \ra X_{2}$ are l.c.i. morphisms of relative codimensions $d$ and $e$, respectively, and $h_{1}: Y \ra X_{1} $ and $h_{2}: Y \ra X_{2} $ are arbitrary morphisms giving the fiber diagram
\[
\xymatrix{
W \ar[d]^{ } \ar[r]^{} &W_{2} \ar[d]^{f_{2}'} \ar[r]^{} &Z_{2} \ar[d]^{f_{2}}\\
W_{1} \ar[d]^{} \ar[r]^{f_{1}'} &Y \ar[d]^{h_{1}} \ar[r]^{h_{2}} &X_{2} \\
Z_{1} \ar[r]^{f_{1}}&X_1  & }
\]
one has $ (f_{1})^{!}_{h_{1}f'_{2}} \circ (f_{2})^{!}_{h_{2}}  = {(f_{2})}^{!}_{h_{2}f'_{1}}  \circ {(f_{1})}^{!}_{h_{1}}: B_{*}Y \ra B_{*-d-e}W$.

$\operatorname{(A_{7})}$ For any smooth morphism $\fyx$ one has $f^{!}_{\id_{X}}=f^{*}$.

For any l.c.i morphism $\fzx$, any morphism $\gyx$ and any morphism $h: Y' \ra Y$, if one forms the fiber diagram 
\[
\xymatrix{
W' \ar[d]^{h'} \ar[r]^{f''} &Y' \ar[d]^h \\
W \ar[d]^{g'} \ar[r]^{f'} &Y \ar[d]^g \\
Z \ar[r]^f&X }
\]
then

$\operatorname{(A_{8})}$ If $g$ and $f$ are \emph{Tor-independent} in $\Schk$ (i.e., if $\operatorname{Tor}_{j}^{\mathcal{O}_{X}}(\mathcal{O}_{Y},\mathcal{O}_{Z})=0$ for all $j>0$) then
	\[
f^{!}_{gh} =	(f')^{!}_{h}
\]

$\operatorname{(A_{9})}$ If $h$ is projective then $f^{!}_{g} \circ h_{*} = h'_{*} \circ f^{!}_{gh} $.

$\operatorname{(A_{10})}$ If $h$ is smooth equidimensional then $f^{!}_{gh} \circ h^{*} = h'^{*} \circ f^{!}_{g} $.

\end{definition}

\begin{notation}
Given a Borel-Moore functor with intersection products $H_{*}$, for any l.c.i. morphism $f: Z\to X$ of relative codimension $d$, the map $f^!_{\id_X}: H_*(X)\to H_{*-d}(Z)$ is called the \emph{l.c.i. pull-back along} $f$ and denoted $f^*$. For each l.c.i morphism $f:Z\ra X$ and each morphism $g: Y \ra X$ we call the morphism $f^{!}_{g}$ \emph{the refined l.c.i. pull-back along} $f$ \emph{associated to} $g$. We will usually denote $f^{!}_{g}$ simply by $f^{!}$ with an indication of where it acts. When the l.c.i. morphism $f$ is a regular embedding then  $f^{!}_{g}$ is called a refined Gysin homomorphism. Refined Gysin homomorphisms and smooth pull-backs can be composed to construct all refined l.c.i. pull-backs.
\end{notation}

%
%

\begin{definition} A \emph{Borel-Moore functor with compatible exterior and intersection products} on $\cC$ consists of a Borel-Moore functor $H_{*}$ on $\cC$ endowed with exterior products and intersection products that in addition satisfy

$\operatorname{(A_{11})}$ If for $i=1$ and $i=2$, $f_{i}:Z_{i}\ra X_{i}$ is an l.c.i. morphism and $g_{i}:Y_{i}\ra X_{i}$ is an arbitrary morphism and one forms the fiber diagram 
\[
\xymatrix{
W_{i} \ar[d]^{g_{i}'} \ar[r]^{f_{i}'} &Y_{i} \ar[d]^{g_{i}} \\
Z_{i} \ar[r]^{f_{i}}&X_{i}, }
\]
one has that for any classes $\alpha_{1} \in H_{*}(Y_{1})$ and $\alpha_{2} \in H_{*}(Y_{2})$
	\[
	(f_{1} \times f_{2})_{g_{1} \times g_{2}}^{!} (\alpha_{1} \times \alpha_{2})   =    (f_{1})_{g_1}^{!}(\alpha_{1}) \times (f_{2})_{g_2}^{!}(\alpha_{2})  \in H_{*} (W_{1} \times W_{2}).
\]
\end{definition}

\begin{notation}
We will call a Borel-Moore functor with compatible exterior and intersection products a {\em refined oriented Borel-Moore pre-homology theory} (ROBM pre-homology theory, for short). 
\end{notation}

Examples of ROBM pre-homology theories are Chow theory $\operatorname{A}_{*}$ (see \cite{FultonIT}), $K$-theory (i.e. the Grothendieck $K$-group functor $\operatorname{G}_{0}$ of the category of coherent $\mathcal{O}_{X}$-modules, graded by $\operatorname{G}_{0} \otimes_{\ZZ} \ZZ[\beta,\beta^{-1}]$ -see \cite[Example 2.2.5]{Levine-Morel}-), and algebraic cobordism $\Omega_{*}$ (see \cite{Levine-Morel}) on the category $\Schk$; and equivariant Chow theory $A_*^G$ and equivariant algebraic cobordism $\Omega_*^G$ on the category $\Schkg$ constructed as in \S\,\ref{section.equivariant.theories} (see \cite{Edidin-Graham}, \cite{Krishna} and \cite{Heller-Malagon}).

Levine and Morel in \cite{Levine-Morel} consider the notion of an \emph{oriented Borel-Moore homology theory}, which is an ROBM pre-homology theory but with l.c.i. pull-backs only instead of refined l.c.i. pull-backs, and with additional axioms called projective bundle, extended homotopy and cellular decomposition properties. Because of the refined l.c.i. pull-backs, an oriented Borel-Moore homology theory is not necessarily an ROBM pre-homology theory. However, one can construct refined l.c.i. pull-backs from ordinary l.c.i. pull-backs by deformation to the normal cone argument of Fulton and MacPherson, provided that the theory additionally satisfies the homotopy and localization properties (see \S\,\ref{section.equivariant.theories} for these properties). We will need the homotopy and localization properties when working with equivariant theories, hence an alternative theory that is sufficient for the constructions below would be a Borel-Moore functor with compatible l.c.i. pull-backs and exterior products, that additionally 
satisfies the homotopy and localization properties.

\begin{definition}      \label{definition.obm.first.chern.class}
 If $H_{*}$ is an ROBM pre-homology theory, for any line bundle $L \ra Y$ in $\mathcal{C}$ with zero section $s:Y\ra L$ one defines the operator $\ch(L):H_{*}(Y)\ra H_{*-1}(Y)$ by $\ch(L)=s^{*}s_{*}$, and calls it the \emph{first Chern class operator of} $L$.
\end{definition}


\subsection{Cohomology theory}\label{sec-cohom}

Let $H_*$ be an ROBM pre-homology theory. For a smooth scheme $X$ of pure dimension $n$, define 
\[ H^*(X) = H_{n-*}(X).\]
For an arbitrary smooth scheme we extend this notion by taking the direct sum over pure dimensional parts of $X$. 

The groups $H^*(X)$ are commutative graded rings with unit, with product defined by 
\begin{alignat*}{4}
 H^*(X) &\times H^*(X) &\to& H^*(X) \\
(a&,b) &\mapsto& \Delta_X^*(a\times b),
\end{alignat*}
where $\Delta_X:X\to X\times X$ is the diagonal map and $\Delta_X^*$ is the l.c.i. pull-back. Associativity of the product follows from $\operatorname{(A_{5})}$ and $\operatorname{(A_{11})}$ applied to two different ways to construct the diagonal $X\to X\times X\times X$ by composing $\Delta_X$.

Let $\pi: X\to \Spec k$ be the structure morphism and define $1_X = \pi^*(1)\in H^0(X)$, where $1= 1_{\Spec k}\in H^0(\Spec k)$ is the element specified in $\operatorname{(D_{3})}$. Then $1_X$ is the multiplicative identity in the ring $H^*(X)$.

Axioms $\operatorname{(A_{5})}$ and $\operatorname{(A_{11})}$ imply that if $f:X\to Y$ is an l.c.i. morphism between smooth schemes, then $f^*: H^*(Y)\to H^*(X)$ is a homomorphism of graded rings with unit. Thus, we may view $H^*$ as a contravariant functor from the category of smooth schemes and l.c.i. morphisms in $\cC$ to the category of commutative graded rings with unit. In the next section we extend this functor to the whole category $\cC$.


\section{Operational Bivariant Theories}   \label{section.bivariant.theories}

In this section we consider a refined oriented Borel-Moore pre-homology theory $B_{*}$ on one of the categories $\mathcal{C}=\Schk$ or $\mathcal{C}=\Schkg$, and associate to it a bivariant theory $B^{*}$ on $\mathcal{C}$. We present a unified treatment of these two cases. Therefore throughout this section we fix one of these two categories and denote it by $\mathcal{C}$, and we assume that all the schemes and morphisms are in $\mathcal{C}$ following the conventions described in \ref{conventions}.  The constructions that we present in this section follow the ideas in \cite[Chapter 17]{FultonIT} where a bivariant theory $A^{*}$ is constructed for the Chow theory $A_{*}$. Some definitions and proofs have been modified to adapt them to our more general setting.


\subsection{Bivariant theories} \label{subsection.operations.bivariant}

A \emph{bivariant theory} $B^*$ on $\cC$ assigns to each morphism $f:X\to Y$ in $\cC$ a graded abelian group $B^*(X\to Y)$. The groups $B^*(X \rightarrow Y)$ are endowed with three operations called product, push-forward and pull-back, which are mutually compatible and admit units:

$\operatorname{(P_{1})}$ \emph{Product}. For all morphisms $f:X\ra Y$ and $g:Y\ra Z$, and all integers $p$ and $q$, there is a homomorphism
\[
B^p(X \xrightarrow{f} Y) \otimes B^q(Y \xrightarrow{g} Z) \stackrel{\cdot}{\longrightarrow} B^{p+q}(X \xrightarrow{gf} Z).
\]
The image of $c\otimes d$ is denoted $c \cdot d$.

$\operatorname{(P_{2})}$ \emph{Push-forward}. If $f:X \ra Y$ is a projective morphism, $g:Y\ra Z$ is any morphism and $p$ is an integer, there is a homomorphism
\[
f_{*}: B^p(X \xrightarrow{gf} Z) \xrightarrow{} B^p(Y \xrightarrow{g} Z).
\]

$\operatorname{(P_{3})}$ \emph{Pull-back}. If $f:X \ra Y$ and $g:Y' \ra Y$ are arbitrary morphisms, $f':X'=X\times_{Y}Y' \ra Y'$ is the projection and $p$ is an integer, there is a homomorphism
\[
g^{*}: B^p(X \xrightarrow{f} Y) \xrightarrow{} B^p(X' \xrightarrow{f'} Y').
\]

$\operatorname{(U)}$ \emph{Units}. For each $X$ there is an element $1_{X} \in B^{0}(X \xrightarrow{\id} X)$, such that $\alpha \cdot 1_{X} = \alpha$ and $1_{X} \cdot \beta = \beta$, for all morphisms $W \ra X$ and $X \ra Y$, and all classes $\alpha \in B^*(W \ra X)$ and $\beta \in B^*(X \ra Y)$. These unit elements are compatible with pull-backs, i.e., $g^{*}(1_{X})=1_{Z}$ for all morphisms $g: Z \ra X$.

These operations are required to satisfy the following seven compatibility properties:

$\operatorname{(B_{1})}$ \emph{Associativity of products}. If $c \in B^{*}(X \ra Y)$, $d \in B^{*}(Y \ra Z)$ and $e \in B^{*}(Z \ra W)$, then
\[
(c \cdot d) \cdot e = c \cdot (d \cdot e) \in B^{*}(X \ra W).
\]

$\operatorname{(B_{2})}$ \emph{Functoriality of push-forwards}. If $c \in B^{*}(X \ra Y)$, then ${\id_{X}}_{*}c =c \in B^{*}(X \ra Y)$. Moreover, if $f:X\ra Y$ and $g:Y\ra Z$ are projective morphisms, $Z \ra W$ is arbitrary, and $d \in B^{*}(X \ra W)$, then
\[
(gf)_{*}d = g_{*}(f_{*}d) \in B^{*}(Z \ra W).
\]

$\operatorname{(B_{3})}$ \emph{Functoriality of pull-backs}. If $c \in B^{*}(X \ra Y)$, then $\id_{Y}^{*}c = c \in B^{*}(X \ra Y)$. Moreover, if $f:X \ra Y$, $g:Y' \ra Y$ and $h: Y'' \ra Y'$ are arbitrary morphisms, $X''=X \times_{Y}Y'' \ra Y''$ is the projection, and $d \in B^{*}(X \ra Y)$, then
\[
(gh)^{*}d = h^{*}(g^{*}d) \in B^{*}(X'' \ra Y'').
\]

$\operatorname{(B_{12})}$ \emph{Product and push-forward commute}. If $f:X \ra Y$ is projective, $Y \ra Z$ and $Z \ra W$ are arbitrary, and $c \in B^{*}(X \ra Z)$ and $d \in B^{*}(Z \ra W)$, then
\[
f_*(c) \cdot d = f_*(c \cdot d) \in B^{*}(Y \ra W).
\]

$\operatorname{(B_{13})}$ \emph{Product and pull-back commute}. If $c \in B^{*}(X \xrightarrow{f} Y)$ and $d \in B^{*}(Y \xrightarrow{g} Z)$, and $h:Z'\ra Z$ is arbitrary, and one forms the fiber diagram
\[
\xymatrix{
X' \ar[d]^{h''} \ar[r]^{f'} &Y' \ar[d]^{h'} \ar[r]^{g'} &Z'\ar[d]^{h}\\
X \ar[r]^f&Y \ar[r]^g &Z,}
\]
then
\[
(h)^{*}(c \cdot d) = h'^*(c)\cdot h^*(d) \in B^{*}(X' \ra Z').
\]

$\operatorname{(B_{23})}$ \emph{Push-forward and pull-back commute}. If $f:X \ra Y$ is projective, $g:Y\ra Z$ and $h:Z'\ra Z$ are arbitrary morphisms, and $c \in B^{*}(X \ra Z)$, and one forms the fiber diagram
\[
\xymatrix{
X' \ar[d]^{h''} \ar[r]^{f'} &Y' \ar[d]^{h'} \ar[r]^{g'} &Z'\ar[d]^{h}\\
X \ar[r]^f&Y \ar[r]^g &Z,}
\]
then
\[
h^*(f_*c) = f'_{*}(h^{*}c) \in B^{*}(Y' \ra Z').
\]

$\operatorname{(B_{123})}$ \emph{Projection formula}. If $f:X \ra Y$ and $g:Y\ra Z$ are arbitrary morphisms, $h':Y'\ra Y$ is projective and $c \in B^{*}(X \ra Y)$ and $d \in B^{*}(Y' \ra Z)$, and one forms the fiber diagram
\[
\xymatrix{
X' \ar[d]^{h''} \ar[r]^{f'} &Y' \ar[d]^{h'} &   \\
X \ar[r]^f&Y \ar[r]^g &Z,}
\]
then
\[
c \cdot h'_{*}(d) = h''_{*}( {h'}^{*}(c) \cdot d) \in B^{*}(X \ra Z).
\]

The group $B^p(X \xrightarrow{f} Y)$ may be denoted by simply by $B^p(X \rightarrow Y)$ or $B^p(f)$. 
We will denote by $B^*(X \xrightarrow{f} Y)$, $B^*(X \rightarrow Y)$ or $B^*(f)$ the direct sum of all $B^p(X \xrightarrow{f} Y)$, for $p \in \mathbf{Z}$.


\subsection{Homology and Cohomology} \label{sec-hom-cohom}

A bivariant theory $B^*(X\to Y)$ contains both a covariant homology theory $B_*(X)$ and a contravariant cohomology theory $B^*(X)$. 

The homology is defined by $B_p(X) = B^{-p}(X\to \Spec k)$. For any projective morphism $f:X\to Y$, the push-forward in the bivariant theory defines the functorial  push-forward map in homology $f_*: B_*(X)\to B_*(Y)$.

The cohomology is defined by $B^p(X) = B^p(\id_X: X\to X)$. The product operation in the bivariant theory turns $B^*(X)$ into a graded ring with unit $1_{X}$ and turns $B_*(X)$ into a graded left module over $B^*(X)$. The product operation $B^*(X)\times B_*(X) \to B_*(X)$ is called cap product and denoted $(\alpha,\beta)\mapsto \alpha\cap\beta$. For any morphism $f: X\to Y$, the pull-back in the bivariant theory defines a functorial pull-back $f^*: B^*(Y) \to B^*(X)$. The pull-back map is a homomorphism of graded rings.  When $f$ is a projective morphism, then the projection formula relates the pull-back, push-forward, and cap product as follows:
\[ f_*(f^*(\alpha)\cap \beta) = \alpha\cap f_*(\beta).\]


\subsection{Operational bivariant theories}\label{sec-op-biv}

We now fix an ROBM pre-homology theory $B_*$ on $\cC$ and associate a bivariant theory $B^*$ to it.

Let $f:X \rightarrow Y$ be any morphism. For each morphism $g:Y' \rightarrow Y$, form the fiber square
\[
\xymatrix{
X' \ar[d]^{g'} \ar[r]^{f'} &Y' \ar[d]^g \\
X \ar[r]^f&Y }
\]
with induced morphisms as labeled.  An element $c$ in $B^{p}(X \xrightarrow{f} Y)$, called a \emph{bivariant class},  is a collection of homomorphisms
\[
c_{g}^{(m)}:B_{m}Y'\rightarrow B_{m-p}X'
\]
for all $g:Y' \rightarrow Y$, and all $m \in \mathbf{Z}$, compatible with projective push-forwards, smooth pull-backs, intersection products and exterior products, i.e.:

$\operatorname{(C_{1})}$ If $h:Y'' \rightarrow Y'$ is projective and $g:Y' \rightarrow Y$ is arbitrary, and one forms the fiber diagram
\[
\xymatrix{
X'' \ar[d]^{h'} \ar[r]^{f''} &Y'' \ar[d]^h \\
X' \ar[d]^{g'} \ar[r]^{f'} &Y' \ar[d]^g \\
X \ar[r]^f&Y }
\]
then for all $\alpha \in B_{m}(Y'')$,
\[
c_{g}^{(m)}(h_{*}\alpha)=h'_{*}c_{gh}^{(m)}(\alpha)
\]
in $B_{m-p}(X')$.

$\operatorname{(C_{2})}$ If $h:Y'' \rightarrow Y'$ is smooth of relative dimension $n$ and $g:Y' \rightarrow Y$ is arbitrary, and one forms the fiber diagram
\[
\xymatrix{
X'' \ar[d]^{h'} \ar[r]^{f''} &Y'' \ar[d]^h \\
X' \ar[d]^{g'} \ar[r]^{f'} &Y' \ar[d]^g \\
X \ar[r]^f&Y }
\]
then for all $\alpha \in B_{m}(Y')$,
\[
c_{gh}^{(m+n)}(h^{*}\alpha)=h'^{*}c_{g}^{(m)}(\alpha)
\]
in $B_{m+n-p}(X'')$.

$\operatorname{(C_{3})}$ If $g:Y' \rightarrow Y$ and $h:Y' \rightarrow Z'$ are morphisms, and $i:Z'' \rightarrow Z'$ is an l.c.i. morphism of codimension $e$, and one forms the fiber diagram
\[
\xymatrix{
X'' \ar[d]^{i''} \ar[r]^{f''} &Y'' \ar[d]^{i'} \ar[r]^{h'} &Z''\ar[d]^i\\
X' \ar[d]^{g'} \ar[r]^{f'} &Y' \ar[d]^g \ar[r]^h &Z' \\
X \ar[r]^f&Y  & }
\]
then for all $\alpha \in B_{m}(Y')$,
\[
c_{gi'}^{(m-e)}(i^{!}\alpha)=i^{!}c_{g}^{(m)}(\alpha)
\]
in $B_{m-e-p}(X'')$.

$\operatorname{(C_{4})}$ If $g:Y' \rightarrow Y$ is arbitrary, and $h:Y' \times Z \rightarrow Y'$ and $h':X' \times Z \rightarrow X'$ are the projections, and one forms the fiber diagram 
\[
\xymatrix{
X'\times Z \ar[d]^{h'} \ar[r]^{f''} &Y'\times Z \ar[d]^h \\
X' \ar[d]^{g'} \ar[r]^{f'} &Y' \ar[d]^g \\
X \ar[r]^f&Y }
\]
then for all $\alpha \in B_{m}(Y')$ and $\beta \in B_{l}(Z)$,
\[
c_{gh}^{(m+l)}(\alpha\times\beta) = c_{g}^{(m)}(\alpha) \times \beta
\]
in $B_{m+l-p}(X'\times Z)$.

The three operations are defined as follows.

Product: Let $c\in B^p(X \xrightarrow{f} Y)$ and $d\in B^q(Y \xrightarrow{g} Z)$. Given any morphism $h:Z' \ra Z$, form the fiber diagram
\[
\xymatrix{
X' \ar[d]^{h''} \ar[r]^{f'} &Y' \ar[d]^{h'} \ar[r]^{g'} &Z'\ar[d]^{h}\\
X \ar[r]^f&Y \ar[r]^g &Z,}
\]
and for each integer $m$ define $(c \cdot d)^{(m)}_{h}=c^{(m-q)}_{h'} \circ d^{(m)}_{h}:B_{m}Z'\ra B_{m-p-q}X'$.

Push-forward: Given $c \in B^p(X \xrightarrow{f} Y \xrightarrow{g} Z)$ and any morphism $h:Z' \ra Z$, form the fiber diagram
\[
\xymatrix{
X' \ar[d]^{h''} \ar[r]^{f'} &Y' \ar[d]^{h'} \ar[r]^{g'} &Z'\ar[d]^{h}\\
X \ar[r]^f&Y \ar[r]^g &Z,}
\]
and for each integer $m$ define $(f_{*}c)^{(m)}_{h}= f'_{*} \circ c^{(m)}_{h}:B_{m}Z'\ra B_{m-p}Y'$.

Pull-back: Given $c \in B^p(X \xrightarrow{f} Y)$ and  morphisms $g: Y'\to Y$ and $h:Y'' \ra Y'$, form the fiber diagram
\[
\xymatrix{
X'' \ar[d]^{h'} \ar[r]^{f''} &Y'' \ar[d]^h \\
X' \ar[d]^{g'} \ar[r]^{f'} &Y' \ar[d]^g \\
X \ar[r]^f&Y, }
\]
and for each integer $m$ define $(g^{*}c)^{(m)}_{h}= c^{(m)}_{gh}:B_{m}Y''\ra B_{m-p}X''$.

It is straightforward to verify that these three operations are well defined (i.e., that $c\cdot d$, $f_*c$ and $g^{*}c$ satisfy $\operatorname{(C_{1})}$-$\operatorname{(C_{4})}$, so that they define classes in the appropriate bivariant groups). Unit elements $1_X \in  B^{0}(X\to X)$, for each $X$, satisfying the property $\operatorname{(U)}$ are defined by letting them act by identity homomorphisms. It is also straightforward to check that the three operations satisfy properties $\operatorname{(B_1)}$-$\operatorname{(B_{123})}$. In conclusion, the operational theory is a bivariant theory.

The only bivariant theories we will consider are the operational ones. By a bivariant theory we will mean an operational bivariant theory associated to an ROBM pre-homology theory.


\subsection{Homology and Cohomology for Operational Bivariant Theories.} \label{subsection.homology.cohomology}

Recall that any bivariant theory $B^*(X\to Y)$ contains a covariant homology theory $B^{-*}(X\to \Spec k)$ and a contravariant cohomology theory $B^*(\id_X : X \rightarrow X)$. We claim that if the bivariant theory $B^*(X\to Y)$ is the operational theory associated to an ROBM pre-homology theory $B_*$, then the homology theory $B^{-*}(X\to \Spec k)$ is isomorphic to the original theory $B_*(X)$. Similarly, the cohomology theory $B^*(X\to X)$ agrees with the cohomology theory $B^*(X)$ constructed in \S\,\ref{sec-cohom} for smooth schemes $X$. The proofs in this section are adapted from the proofs in \cite{FultonIT} for the Chow theory.

\begin{proposition} \label{proposition.homology}
For any $X$ and each integer $p$ the homomorphism 
\[
\varphi : B^{-p}(X \ra \operatorname{Spec}k) \ra B_{p}(X)
\]
taking a bivariant class $c$ to $c(1)$ is an isomorphism. Here $1= 1_{\Spec k}\in B_0(\Spec k)$ is the element specified by $\operatorname{(D_{3})}$ in Definition~\ref{def-ext-prod}. The isomorphism $\varphi$ is natural with respect to push-forwards along projective morphisms.
\end{proposition}
\begin{proof}
Define a homomorphism $\psi : B_{p}(X) \ra B^{-p}(X \ra \operatorname{Spec}k)$ as follows: Given any $a \in B_{p}(X)$, any morphism $f: Y \ra \operatorname{Spec}k$ and a class $\alpha \in B_{m}(Y)$, we define $\psi(a)(\alpha)=a \times \alpha \in B_{m+p}(X \times Y)$. It follows at once that $\psi(a)$ satisfies $\operatorname{(C_{1})}$-$\operatorname{(C_{4})}$ and $\psi$ is a well defined homomorphism. 

For each $a \in B_{p}(X)$, one has $\varphi(\psi(a))=\psi(a)(1)=a\times 1=a \in B_{p}(X)$, so $\varphi \circ \psi$ is the identity. Given any $c \in B^{-p}(X \ra \operatorname{Spec}k)$, any morphism $Y \ra \operatorname{Spec}k$ and any class $\alpha \in B_{m}(Y)$, one has $\psi(\varphi(c))(\alpha)=\psi(c(1))(\alpha)= c(1) \times \alpha = c(1 \times \alpha) = c(\alpha)  \in B^{m+p}(X \times Y)$, so $\psi \circ \varphi$ is also the identity.  

Naturality of $\varphi$ with respect to projective push-forwards follows from the definition of push-forward in the operational bivariant theory.
\end{proof}

Let us now consider the cohomology theory $B^*(X\to X)$. Note that an element $c \in B^{p}(X\to X)$ is a collection of homomorphisms $c_{f}^{(m)}:B_{m}X'\ra B_{m-p}X'$, for all morphisms $f: X'\ra X$ and all integers $m$, that are compatible with projective push-forwards, smooth pull-backs, exterior and intersection products. Using the previous proposition we identify $B_*(X)$ with $B^{-*}(X\to\Spec k)$ and thus give  $B_*(X)$ the structure of a module over $B^*(X\to X)$.

\begin{proposition}[Poincar\'e duality] \label{poincare} Let $X$ be a smooth purely $n$-dimensional scheme and let $B^*(X) = B_{n-*}(X)$ be the cohomology theory defined in \S\,\ref{sec-cohom}. The homomorphism defined by cap product with $1_X\in B^0(X)$:
\[
\varphi: B^{*}(X \xrightarrow{\id_X} X) \xrightarrow{\cap 1_{X}} B^*(X)
\]
is an isomorphism of graded rings that takes $1_X\in B^{0}(X\to X)$ to $1_X\in B^0(X)$. The isomorphism $\varphi$ is natural with respect to pull-backs by l.c.i. morphisms.
\end{proposition}
\begin{proof}
Let us fix an integer $p$ and define a homomorphism $\psi : B^{p}(X)=B_{n-p}(X) \ra B^{p}(X\to X)$ as follows: Given any $a \in B_{n-p}(X)$, any morphism $f: Y \ra X$ and a class $\alpha \in B_{m}(Y)$, we define $\psi(a)(\alpha)=\psi(a)^{(m)}_{f}(\alpha)= \gamma^{*}_{f}(a\times \alpha) \in B_{m-p}Y$, where $\gamma_{f} =(f,\id_Y):Y\ra X \times Y$ is the transpose of the graph of $f$, which in this case is a regular embedding of codimension $n$. It is straightforward to check that $\psi(a)$ is a bivariant class and $\psi$ is a group homomorphism.

For each $a \in B_{n-p}X$, one has $\varphi(\psi(a))= \psi(a)(1_{X})= \gamma_{\id_{X}}^{*}( a\times 1_{X} ) = \id^{*}_X a =a \in B_{n-p}X$, so $\varphi \circ \psi$ is the identity. Given any $c \in B^{p}(X\to X)$, any morphism $f: Y \ra X$ and any class $\alpha \in B_{m}Y$, one has $\psi(\varphi(c))(\alpha)=\psi(c(1_{X}))(\alpha)=\gamma_{f}^{*}(\
 c(1_{X})\times\alpha)=\gamma_{f}^{*}(c(1_{X}\times\alpha))=c(\gamma_{f}^{*}( 1_{X}\times\alpha))=c(\alpha) \in B_{m-p}Y$, so $\psi \circ \varphi$ is also the identity.

To verify the compatibility with multiplication, we show that for arbitrary classes $a \in B_{n-p}(X)$ and $b \in B_{n-q}(X)$ we have $\psi(a \cdot b) = \psi(a) \cdot \psi(b) \in B^{p+q}(X\to X)$. Indeed, given any morphism $\fyx$ and any class $\theta \in B_{m}Y$, we have
\begin{align*}
	\psi(a \cdot b)(\theta) 
			&=\gamma_{f}^{*}(\gamma_{\id_{X}}^{*}(a \times b)\times \theta) 
					=  (f,f,\id_{Y})^{*} (a \times b \times\theta)      \\
			&= \gamma_{f}^{*}((\id_X,\gamma_{f})^{*}(a\times b\times\theta))
					= \gamma_{f}^{*}(a\times \gamma_{f}^{*}(b \times \theta))      \\
			&		=\psi(a)(\psi(b)(\theta)) =(\psi(a) \cdot \psi(b))(\theta) \in B_{m-p-q}Y. 
\end{align*}

The element $1_X\in B^{0}(X\to X)$ acts as multiplicative identity, hence cap product with it maps $1_X \in  B^0(X)$ to itself.

Let $f:X\to Y$ be an l.c.i. morphism between pure dimensional smooth schemes. Naturality of $\varphi$ with respect to pull-back by $f$ is equivalent to the identity
\[ f^*(c)\cap 1_X = f^*(c\cap 1_Y)\]
for any $c\in B^*(Y\to Y)$, which holds by $\operatorname{(C_{3})}$ since $1_X=f^*(1_Y)$. 
\end{proof}

\begin{notation}
We will denote $B^*(X) = B^*(\id_X : X \rightarrow X)$ for an arbitrary scheme $X$ in $\cC$. Proposition~\ref{poincare} shows that this contravariant functor on $\cC$, when restricted to the category of smooth schemes and l.c.i. morphisms, is isomorphic to the functor $B^*$ defined in \S\,\ref{sec-cohom}.   
\end{notation}

\begin{remark}
In his construction of a bivariant theory associated to Chow theory in \cite[Chapter 17]{FultonIT}, Fulton uses the Chow theory versions of our axioms $\operatorname{(C_{1})}$, $\operatorname{(C_{2})}$ and $\operatorname{(C_{3})}$, namely, compatibility of the bivariant classes with proper push-forwards, flat pull-backs and refined Gysin homomorphisms; but there is no explicit requirement of our axiom $\operatorname{(C_{4})}$, compatibility with exterior products. This axiom $\operatorname{(C_{4})}$ does not appear explicitly in Fulton's construction because in the case of Chow theory $A_{*}$ axioms $\operatorname{(C_{1})}$ and $\operatorname{(C_{2})}$ imply axiom $\operatorname{(C_{4})}$. More generally, this is true for any ROBM pre-homology theory $B_{*}$ that satisfies that for every $X$ in $\mathcal{C}$ the group $B_{*}(X)$ is generated by the projective push-forward images of the
classes $1_{Y}$ for all smooth varieties $Y$ with a projective map to $X$. Indeed, using $\operatorname{(C_{1})}$ one reduces $\operatorname{(C_{4})}$ to the case where both $Y'$ and $Z$ are smooth varieties, and $\alpha=1_{Y'}$ and $\beta=1_{Z}$. In that case, from $\operatorname{(C_{2})}$ one obtains that 
$c_{gh}^{(m+l)}(1_{Y'} \times 1_{Z}) 
=c_{gh}^{(m+l)}(1_{Y'\times Z})
=c_{gh}^{(m+l)}(h^{*}1_{Y'})
= h'^{*}c_{g}^{(m)}(1_{Y'})
= c_{g}^{(m)}(1_{Y'}) \times 1_{Z}$, and then $\operatorname{(C_{4})}$ holds in general for $B_{*}$. 
\end{remark}


\section{The Equivariant Version of an ROBM pre-Homology Theory}     \label{section.equivariant.theories}

In this section we fix an ROBM pre-homology theory $B_{*}$ on the category $\Schk$ and construct its \emph{equivariant version} $B_*^G$, which is an ROBM pre-homology theory on $\Schkg$. The construction of $B_*^G$ generalizes to arbitrary ROBM pre-homology theories similar constructions in Chow theory by Totaro \cite{Totaro} and Edidin-Graham \cite{Edidin-Graham}, and in algebraic cobordism by Krishna \cite{Krishna} and Heller-Malag\'on-L\'opez \cite{Heller-Malagon}.

We will need to assume throughout this section that the field $k$ is infinite and the theory $B_*$ satisfies the homotopy property $\operatorname{(H)}$ and the localization property $\operatorname{(L)}$: 

$\operatorname{(H)}$ Let $p: E \ra X$ be a vector bundle of rank $r$ over $X$ in $\Schk$. Then $p^{*}: B_{*}(X) \ra B_{*+r}(E)$ is an isomorphism. 

$\operatorname{(L)}$ For any closed immersion $i: Z \ra X$ with open complement $j: U= X \smallsetminus Z \ra X$ the following sequence is exact:
\[ 
B_{*}(Z) \xrightarrow{i_{*}} B_{*}(X) \xrightarrow{j^{*}} B_{*}(U) \xrightarrow{} 0.
\]



\subsection{Algebraic Groups, Quotients and Good Systems of Representations}
\label{quotients} 

Let $G$ be a linear algebraic group. If $X$ is a scheme with a $G$-action $\sigma: G \times X \ra X$ and the geometric quotient of $X$ by the action of $G$ exists it will be denoted by $X \ra X/G$. When the geometric quotient $\pi: X \ra X/G$ exists, it is called a \emph{principal} $G$\emph{-bundle} if the morphism $\pi$ is faithfully flat and the morphism $\psi=(\sigma,\operatorname{pr_{X}}):G \times X \ra X \times_{X / G} X$ is an isomorphism. By \cite[Proposition 0.9]{GIT}, if $G$ acts freely on $U \in \Schkg$ and the geometric quotient $\pi: U \ra U/G$ exists in $\Schk$, for some quasiprojective scheme $U/G$, then $\pi: U \ra U/G$ is a principal $G$-bundle. Moreover, by \cite[Proposition 7.1]{GIT}, for any $X \in \Schkg$ the geometric quotient $\pi': X \times U \ra (X \times U) / G$ also exists in $\Schk$, it is a principal $G$-bundle and $(X \times U) / G$ is quasiprojective. In this case we denote the scheme $(X \times U) / G$ by $X \times^{G} U$.

\begin{definition}
We say that $\{(V_{i},U_{i})\}_{i \in \mathbf{Z}^{+}}$ is a \emph{good system of representations} of $G$ if each $V_{i}$ is a $G$-representation, $U_{i} \subseteq V_{i}$ is a $G$-invariant open subset and they satisfy the following conditions:
	\begin{enumerate}
		\item $G$ acts freely on $U_{i}$ and $U_{i}/G$ exists as a quasiprojective scheme in $\Schk$.

		\item For each $i$ there is a $G$-representation $W_{i}$ so that $V_{i+1}= V_{i} \oplus W_{i}$.

		\item $U_{i} \oplus \{ 0 \} \subseteq U_{i+1}$ and the inclusion factors as $U_{i} = U_{i} \oplus \{ 0 \} \subseteq U_{i} \oplus W_{i} \subseteq  U_{i+1}$.


		\item $\operatorname{codim}_{V_{i}}(V_{i}\smallsetminus U_{i}) < \operatorname{codim}_{V_{j}}(V_{j}\smallsetminus U_{j})$, for $i<j$.
	\end{enumerate}
\end{definition}

For any algebraic group $G$ there exist good systems of representations (see \cite[Remark 1.4]{Totaro}). The following lemma lists some basic facts regarding the properties of morphisms induced on geometric quotients.

\begin{lemma} \label{lemma.properties.morphisms} 
Let $f:X\ra Y$ be a $G$-equivariant morphism in $\Schkg$ and let $\{(V_{i},U_{i})\}$ be a good system of representations of $G$.
\begin{enumerate}
\item For each $i$ the quotient $X \times^{G} U_{i}=(X \times U_{i})/G$ exists in $\Schk$ and it is quasiprojective. The induced morphisms $\phi_{ij}:X \times^{G} U_{i} \ra X \times^{G} U_{j}$ are l.c.i. morphisms for $j \geq i$. If $X$ is smooth then $X \times^{G} U_{i}$ is also smooth.
\item Let $\mathbf{P}$ be a property of morphisms in the following list: open immersion, closed immersion, regular embedding, projective, smooth, l.c.i. If $f:X\ra Y$ satisfies the property $\mathbf{P}$ in the category $\Schkg$ then the induced maps $f_{i}: X \times^{G} U_{i} \ra Y \times^{G} U_{i}$ satisfy property $\mathbf{P}$ in the category $\Schk$. 
\item For any morphisms $g: Y \ra X$ and $f: Z \ra X$ in $\Schkg$ and any indices $j \geq i$, the square diagrams 
\[
\xymatrix{
W\times U_{i} \ar[d]^{} \ar[r]^{} &Y\times U_{j} \ar[d]^{} \\
Z\times U_{i} \ar[r]^{}&X \times U_{j}, }
\qquad \qquad
\xymatrix{
W\times^{G} U_{i} \ar[d]^{} \ar[r]^{} &Y\times^{G} U_{j} \ar[d]^{g'} \\
Z\times^{G} U_{i} \ar[r]^{f'} &X\times^{G} U_{j}, }
\]
induced by the Cartesian product $W=Y \times_{X} Z$ are fiber squares, and furthermore they are Tor-independent if $f$ and $g$ are Tor-independent. 
\end{enumerate}
\end{lemma}
\begin{proof}
For proofs of the assertions in $(1)$ and $(2)$ see \cite[Proposition 2]{Edidin-Graham} and \cite[2.2.2, Lemma~9]{Heller-Malagon}). For $(3)$, given any $T \in \Schk$ and morphisms $g'': T \ra Y \times^{G} U_{j}$ and $f'':  T \ra Z\times^{G} U_{i}$ such that $f' \circ f'' = g' \circ g''$, we let $G$ act on $\tilde{T}= T \times_{(X \times^{G} U_{j})} (X \times U_{j})$ via the morphism $G \times \tilde{T} \ra \tilde{T}$ induced by the product of the trivial action of $G$ on $T$ and the action of $G$ on $X \times U_{j}$. By \cite[Amplification~7.1]{GIT}, $G$-equivariant morphisms from $\tilde{T}$ to each of $X \times U_{j}$, $Y \times U_{j}$, $Z \times U_{i}$ and $W \times U_{i}$ correspond to the morphisms induced on the quotients from $T$ to $X \times^{G} U_{j}$, $Y \times^{G} U_{j}$, $Z \times^{G} U_{i}$ and $W \times^{G} U_{i}$, respectively. The assertion that the squares in $(3)$ are Cartesian follows easily from these observations. If $f$ and $g$ are Tor-independent, then $\operatorname{Tor}_{m}^{\mathcal{O}_{X}}(\mathcal{O}_{Y},\mathcal{O}_{Z})=0$ for all $m>0$, and clearly $\operatorname{Tor}_{m}^{\mathcal{O}_{U_j}}(\mathcal{O}_{U_j},\mathcal{O}_{U_i})=0$ for all $m>0$. Hence, by applying locally the spectral sequence associated to the double complex obtained as the tensor product of two complexes, we obtain that $\operatorname{Tor}_{m}^{\mathcal{O}_{X \times U_j}}(\mathcal{O}_{Y \times U_j},\mathcal{O}_{Z \times U_i})=0$ for all $m>0$. Since $Y \times U_{j} = (Y \times^{G} U_{j}) \times_{(X \times^{G} U_{j})} (X \times U_{j})$ and $Z \times U_{i} = (Z \times^{G} U_{i}) \times_{(X \times^{G} U_{j})} (X \times U_{j})$, by faithfully flat base change for $\operatorname{Tor}_{m}$, we have that $\operatorname{Tor}_{m}^{\mathcal{O}_{X \times^{G} U_j}}(\mathcal{O}_{Y \times^{G} U_j},\mathcal{O}_{Z \times^{G} U_i})=0$ for all $m>0$. Therefore, in this case the given squares are also Tor-independent.
\end{proof}


\subsection{Construction of $B_{*}^G(X)$}     \label{construction.induced.equivariant.OBM}
Fix a good system of representations $\{ (V_{i},U_{i}) \}$ of $G$. 
For any scheme $X \in \Schkg$ define $B_{*}^{G}(X) =\bigoplus_{n \in \mathbf{Z}} B_{n}^{G}(X)$, where 
\[
B_{n}^{G}(X)=\varprojlim_{i}B_{n+\operatorname{dim}U_{i}-\operatorname{dim}G}(X \times^{G} U_{i}).
\]
To simplify notation, we will often write 
\[ B_{*}^{G}(X)  = \varprojlim_{i}B_*(X \times^{G} U_{i}),\]
where the limit is taken in each degree separately. Equivalently, the limit is taken in the category of graded abelian groups, with $B_*(X \times^{G} U_{i})$ having grading shifted so that the maps in the inverse system are homogeneous of degree zero.

To see that $B^{G}_{*}$ is independent of the choice of a good system of representations, one can formally follow the argument presented in the case of algebraic cobordism in \cite[Proposition 15 and Theorem 16]{Heller-Malagon}, which we outline below for the reader's convenience. The proof of the next proposition requires the field $k$ to be infinite.

\begin{proposition} \label{proposition.homotopy.refined}
Assume that the ROBM pre-homology theory $B_{*}$ satisfies the properties $\operatorname{(H)}$ and $\operatorname{(L)}$. Let $\pi:E\ra X$ be a vector bundle over a scheme $X$ of rank $r$. Let $U \subseteq E$ be an open subscheme with closed complement $S=E\smallsetminus U$.
\begin{enumerate}
\item If $X$ is affine and $\operatorname{codim}_{E}S > \operatorname{dim}X $ then $\pi|_{U}^{*}: B_{m}(X) \ra B_{m+r}(U)$ is an isomorphism for all $m$. 
\item \label{proposition.homotopy.refined.2} For $X$ arbitrary, there is as integer $n(X)$ depending only on $X$, such that $\pi|_{U}^{*}: B_{m}(X) \ra B_{m+r}(U)$ is an isomorphism for all $m$ whenever $\operatorname{codim}_{E}S > n(X) $.
\end{enumerate}
\end{proposition}
\begin{proof}
The case when $B_{*}$ is algebraic cobordism is \cite[Proposition 15]{Heller-Malagon}. The proof given in \cite{Heller-Malagon} only uses the formal properties of algebraic cobordism as an ROBM pre-homology theory satisfying $\operatorname{(H)}$ and $\operatorname{(L)}$, so it translates formally to the present setting.
\end{proof}

\begin{proposition}
For any $X \in \Schkg$,  $B^{G}_{*}(X)$ is independent of the choice of a good system of representations of $G$ up to canonical isomorphism. 
\end{proposition}
\begin{proof}
We use Bogomolov's double filtration argument.
 Let $\{ (V_{i},U_{i}) \}$ and $\{ (V'_{i},U'_{i}) \}$ be good systems of representations of $G$. For a fixed index $i$, since $G$ acts freely on $U_{i}$, it also acts freely on $U_{i} \times V'_{j}$ and $U_{i} \times U'_{j}$ for all $j$. Hence $X \times^{G} (U_{i} \times V'_{j}) \ra X \times^{G} U_{i}$ is a vector bundle. By Proposition~\ref{proposition.homotopy.refined} (\ref{proposition.homotopy.refined.2}), there is an integer $N_{i}$ such that the l.c.i. pull-backs induce canonical isomorphisms 
 \[B_*(X \times^{G} U_{i}) \cong B_*(X \times^{G} (U_{i} \times U'_{j}))
\] 
for each $j \geq N_{i}$. Therefore, there is a canonical isomorphism
 \begin{equation} \label{eqn.limit.1}
\varprojlim_{i}B_*(X \times^{G} U_{i}) =\varprojlim_{i}\varprojlim_{j}B_*(X \times^{G} (U_{i}\times U'_{j})).
 \end{equation}
Similarly, exchanging the role of the good systems of representations we obtain a canonical isomorphism 
  \begin{equation} \label{eqn.limit.2}
\varprojlim_{j}B_*(X \times^{G} U'_{j}) =\varprojlim_{j}\varprojlim_{i}B_*(X \times^{G} (U_{i}\times U'_{j})).
 \end{equation}
 To get the conclusion we only need to observe that the right sides of (\ref{eqn.limit.1}) are (\ref{eqn.limit.2}) canonically isomorphic to the inverse limit of the system $\{ B_*(X \times^{G} (U_{i}\times U'_{j}))  \}_{i,j}$, where the maps 
 $$B_*(X \times^{G} (U_{i}\times U'_{j})) \ra B_*(X \times^{G} (U_{i'}\times U'_{j'}))$$ are the corresponding l.c.i. pull-backs for all $i \geq i'$ and $j\geq j'$.
\end{proof}


\subsection{The Induced ROBM pre-Homology Theory Structure on $B^{G}_{*}$}

We now show that the ROBM pre-homology structure of the theory $B_{*}$ induces an ROBM pre-homology structure on $B^{G}_{*}$. We use functoriality of the inverse limit to construct  projective push forwards, smooth pull-backs, exterior and intersection products on $B^G_{*}$. To define a homomorphism between two inverse limits, we construct a map between the two inverse systems.


Fix a good system of representations $\{ (V_{i},U_{i}) \}$ of $G$. Given any projective $G$-equivariant morphism $\fxy$, the morphisms $\{ f_{i}: X \times^{G} U_{i} \ra Y \times^{G} U_{i} \}$ are projective. The fiber square on the left in (\ref{constructing.push}) is Tor-independent for any $j \geq i$, 
\begin{equation} \label{constructing.push}
\xymatrix{
 X \times^{G} U_{i} \ar[d] \ar[r] &  Y \times^{G} U_{i} \ar[d] \\
  X \times^{G} U_{j}  \ar[r] &  Y \times^{G} U_{j} 
}
\qquad \qquad \quad
\xymatrix{
 B_*(X \times^{G} U_{i})  \ar[r]^{{f_i}_*} &  B_*(Y \times^{G} U_{i})   \\
  B_*(X \times^{G} U_{j})  \ar[u] \ar[r]^{{f_j}_*} &  B_*(Y \times^{G} U_{j}) \ar[u] 
}
\end{equation}
hence the square on the right in (\ref{constructing.push}) is commutative for any $j \geq i$. Hence, the maps ${f_i}_*$ induce a homomorphism between the limits $f_*:  B^G_*(X) \to B^G_*(Y)$. 

Smooth pull-backs are defined in a similar way. For intersection products, given a $G$-equivariant l.c.i morphism $\fzx$ of codimension $d$ and any $G$-equivariant morphism $\gyx$, with $W= Z \times_{X} Y$, first we apply the operation $\times^G U_i$ to the whole intersection product diagram. The result is again an intersection product diagram. The two leftmost fiber squares in (\ref{constructing.lci}) are Tor-independent for any $j \geq i$, 
\begin{equation} \label{constructing.lci}
\xymatrix{
 Z \times^{G} U_{i} \ar[d] \ar[r]^{f_i} &  X \times^{G} U_{i} \ar[d] \\
  Z \times^{G} U_{j}  \ar[r]^{f_j} &  X \times^{G} U_{j} 
}
\quad 
\xymatrix{
 W \times^{G} U_{i} \ar[d] \ar[r]^{} &  Y \times^{G} U_{i} \ar[d] \\
  W \times^{G} U_{j}  \ar[r]^{} &  Y \times^{G} U_{j} 
}
\quad
\xymatrix{
 B_{*-d}(W \times^{G} U_{i})   &  B_*(Y \times^{G} U_{i}) \ar[l]_<<<{f_i^!}   \\
  B_{*-d}(W \times^{G} U_{j})  \ar[u]  &  B_*(Y \times^{G} U_{j}) \ar[u] \ar[l]_<<<{f_j^!} 
}
\end{equation}
hence the square on the right in (\ref{constructing.lci}) is commutative for any $j \geq i$. Hence, the maps $f_i^!$ induce a homomorphism between the limits $f^!:  B^G_*(Y) \to B^G_{*-d}(W)$.

To define the exterior product, note that the morphisms 
\[\phi_{i}: (X \times Y)  \times^{G}  (U_{i} \times U_{i}) \ra (X \times^{G} U_{i})  \times  (Y \times^{G} U_{i}) \] 
are smooth (see \cite[Theorem 6.8]{Borel}). We compose the associated smooth pull-back and the exterior product of $B_{*}$ to get
\[ \times_i: B_*(X \times^{G} U_{i})  \times  B_*(Y \times^{G} U_{i})  \to B_*((X \times^{G} U_{i})  \times  (Y \times^{G} U_{i})) \to  B_*((X \times Y)  \times^{G}  (U_{i} \times U_{i})).\]
The morphisms $\times_i$ are compatible with maps in the inverse systems, hence they define the exterior product map between limits:
 \[ \times : B_{*}^{G}X \times B_{*}^{G}Y \ra  B_{*}^{G}(X \times Y).\]
The elements $1_{U_i/G} \in B_*(\Spec k \times^G U_i)$ define $1\in B_*^G(X)$.

It remains to prove that the theory $B_*^G$ with the operations defined above satisfies the axioms of an ROBM pre-homology theory. Each axiom amounts to a statement about the commutativity of a diagram  of homomorphisms. One can check that in each case the commutativity holds at each level $i$, hence it also holds in the limit.
Using the double filtration argument as before, one can also check that these projective push-forwards, smooth pull-backs, exterior and intersection products are independent of the good system of representations. 

The conclusions of this section can then be summarized as:

\begin{theorem}
The functor $B^G_{*}$ with the projective push-forwards, smooth pull-backs, exterior and intersection products constructed above is a refined oriented Borel-Moore pre-homology theory on the category $\Schkg$. We call $B^G_{*}$ the \emph{equivariant version} of $B_{*}$.
\end{theorem}


\subsection{Operational Equivariant Theory}  \label{question.oe.eo}

For a given ROBM pre-homology theory $B_{*}$ on $\Schk$ we constructed an associated equivariant version $B^{G}_{*}$ as an ROBM pre-homology theory on $\Schkg$. The construction of \S\,\ref{section.bivariant.theories} applied to $B_{*}$ and to $B^{G}_{*}$ produces respectively operational bivariant theories  $B^*$ on $\Schk$ and $(B^{G})^{*}$ on $\Schkg$. We denote $(B^{G})^{*}$ by $B^*_{G}$ and call it the \emph{operational equivariant version} of $B_{*}$

One can switch the order of the two steps in the construction of $B_G^*$ and  define an "\emph{equivariant operational}" theory 
 \[ 
\tilde{B}^*_G(X) = \varprojlim_{i} B^*(X \times^{G} U_{i}).
\]
The two theories $B^*_G(X)$ and $\tilde{B}^*_G(X)$ turn out to be isomorphic if we assume that $B_{*}$ satisfies the descent property $\operatorname{(D)}$ described in the next section. This property was first proved by Gillet \cite{GilletSeq} in the case of Chow groups and $K$-theory. It has several other consequences for ROBM pre-homology theories that are studied in the next section.


\section{Descent Sequences}                     \label{section.kimura.bivariant}

We assume in this section that the field $k$ has characteristic zero, or more generally, we assume that every scheme $X$ in $\Schk$ or  in $\Schkg$ has a smooth projective (equivariant) envelope $\pi: \tilde{X}\to X$ as defined in \ref{envelopes}.
We fix an ROBM pre-homology theory $B_{*}$ on one of the categories $\Schk$ or $\Schkg$ and consider the following property $\operatorname{(D)}$:

$\operatorname{(D)}$ For any envelope $\pi:\tilde{X}\ra X$, with $\pi$ projective, the following sequence is exact
\begin{equation} \label{sequence.envelopes}
 B_{*}(\tilde{X} \times_{X} \tilde{X}) \xrightarrow{p_{1*}-p_{2*}}    B_{*}(\tilde{X})    \xrightarrow{\pi_{*}}   B_{*}(X)    \xrightarrow{}  0,
\end{equation}
where $p_{i}:\tilde{X} \times_{X} \tilde{X} \ra \tilde{X}$ is the projection on the $i^{\textnormal{th}}$ factor, for $i=1,2$.


\subsection{Envelopes} \label{envelopes} An \emph{envelope} of a scheme $X$ in $\Schk$ is a proper morphism $\pi: \tilde{X} \ra X$ such that for every subvariety $V$ of $X$ there is a subvariety $\tilde{V}$ of $\tilde{X}$ that is mapped birationally onto $V$ by $\pi$. If $G$ is an algebraic group, a $G$-\emph{equivariant envelope} of a scheme $X$ in $\Schkg$ is a proper $G$-equivariant morphism $\pi: \tilde{X} \ra X$ such that for every $G$-invariant subvariety $V$ of $X$ there is a $G$-invariant subvariety $\tilde{V}$ of $\tilde{X}$ that is mapped birationally onto $V$ by $\pi$. 

In the following, an envelope in the category $\Schk$ means an ordinary envelope and an envelope in the category $\Schkg$ means a $G$-equivariant envelope. If $\pi: \tilde{X} \ra X$ is an envelope, we say that it is a \emph{smooth envelope} if $\tilde{X}$ is smooth, and we say that it is a \emph{projective envelope} if $\pi$ is a projective morphism. Likewise, we say that the envelope $\pi: \tilde{X} \ra X$ is \emph{birational} if for some dense open subset $U$ of $X$, $\pi$ induces an isomorphism $\pi|:\pi^{-1}(U) \ra U$. 
The composition of envelopes is again an envelope and the fiber product of an envelope by any morphism is again an envelope.

The domain $\tilde{X}$ of an envelope is not required to be connected, hence if we assume that varieties over $k$ admit (equivariant) resolutions of singularities via a projective morphism, then by induction on the dimension it follows easily that for every scheme $X$ in $\Schk$ (respectively in $\Schkg$) there exists a smooth projective birational (equivariant) envelope $\pi: \tilde{X} \ra X$.

Notice that if $\pi: \tilde{X} \ra X$ is an envelope in $\Schkg$ and if $\tilde{U} \in \Schkg$ has a free $G$-action such that $\tilde{U}/G$ exists as a quasiprojective scheme in $\Schk$, then the induced morphism $\pi_{G}: \tilde{X} \times^{G} \tilde{U} \ra X \times^{G} \tilde{U}$ is an envelope in $\Schk$. Furthermore if $\pi$ is either a smooth, projective or birational envelope, then $\pi_{G}$ is a smooth, projective or birational envelope, respectively. Moreover, if $\pi|: \pi^{-1}(U) \ra U$ is an isomorphism for some $G$-invariant open subset $U$ of $X$ and $\{ Z_{i} \}$ are $G$-invariant closed subschemes of $X$ such that $X \smallsetminus U = \cup Z_{i}$, then $\pi_{G}$ maps the open subset $\pi_{G}^{-1}( U \times^{G} \tilde{U}) = \pi^{-1}(U) \times^{G} \tilde{U}$ of $\tilde{X} \times^{G} 
\tilde{U}$ isomorphically onto the open subset $ U \times^{G} \tilde{U}$ of $X \times^{G} \tilde{U}$, and the closed subschemes $\{ Z_{i} \times^{G} \tilde{U} \}$ of $X \times^{G} \tilde{U}$ satisfy that $(X \smallsetminus U ) \times^{G} \tilde{U} = (X \times^{G} \tilde{U}) \smallsetminus (U \times^{G} \tilde{U} ) = \cup (Z_{i} \times^{G} \tilde{U})$.


\subsection{Operational Equivariant vs. Equivariant Operational} \label{subsection.OEvsEO}

Assume now that the theory $B_*$ on $\Schk$ satisfies properties $\operatorname{(H)}$ and $\operatorname{(L)}$, and  we can thus define the operational equivariant theory $B^*_G$ as well as the "equivariant operational" theory $\tilde{B}^*_G$ as in \S\,\ref{question.oe.eo}. We show that if $B_*$ also satisfies property $\operatorname{(D)}$, then these two bivariant theories are isomorphic. For simplicity, we prove this isomorphism only for the bivariant cohomology theory $B^*_G(X)$.

\begin{lemma} \label{surjectivity.equivariant.envelopes} Let $B_{*}$ be an ROBM pre-homology theory on $\Schk$ that satisfies $\operatorname{(H)}$, $\operatorname{(L)}$ and $\operatorname{(D)}$. Then, for any scheme $X \in \Schkg$ and any projective envelope $\pi: \tilde{X}\ra X$ the push-forward homomorphism $\pi_{*}: B^{G}_*(\tilde{X}) \ra B^{G}_*(X)$ is surjective.
\end{lemma}
\begin{proof}

Recall that the inverse limit $\varprojlim$ is a left exact functor from the category of inverse systems of (graded) abelian groups. Given a short exact sequence of inverse systems
\[ 0\to (E_i) \to (F_i) \to (G_i)\to 0,\]
the sequence of limits 
\[ 0\to \varprojlim E_i \to \varprojlim F_i \to \varprojlim G_i\to 0\]
is exact if the system $(E_i)$ satisfies the Mittag-Leffler condition. This condition is satisfied for example if the maps $E_i\to E_j$ for $i>j$ in the inverse system are all surjective. 

Given an exact sequence
 \[  (E_i) \to (F_i) \to (G_i)\to 0,\]
one may replace the system $(E_i)$ with its image $(I_i)$ in $(F_i)$ to get a short exact sequence. If all maps in the inverse system $(E_i)$ are surjective, then they are also surjective in $(I_i)$, and it follows that the map of limits  $\varprojlim F_i \to \varprojlim G_i$ is surjective.

We apply the previous discussion to the sequence of inverse systems:
 \[
 B_*((\tilde{X}\times_X\tilde{X})\times^{G} U_{i})
 \rightarrow B_*(\tilde{X}\times^{G} U_{i}) \rightarrow B_*(X\times^{G} U_{i}) \rightarrow 0.
 \]
 This sequence is exact by property $\operatorname{(D)}$: the map $\tilde{X}\times^{G} U_{i} \to X\times^{G} U_{i}$ is an envelope and 
 \[ (\tilde{X}\times_X\tilde{X})\times^{G} U_{i} \isom (\tilde{X}\times^G U_i) \times_{X \times^G U_i}  (\tilde{X} \times^{G} U_{i}).\]
 
 For any scheme $Y$ in $\Schkg$, the l.c.i. pull-backs 
 \[ B_*(Y\times^G U_i) \to B_*(Y\times^G U_j)\]
 are surjective for all $i>j$. This follows from properties $\operatorname{(H)}$ and $\operatorname{(L)}$ because the inclusion $Y \times^{G}   U_{j} \ra Y \times^{G}   U_{i}$ can be factored as the inclusion of the zero section of a vector bundle followed by an open immersion. Applying this to the case $Y=\tilde{X}\times_X\tilde{X}$ gives the statement of the lemma.
 \end{proof}

\begin{proposition}         \label{proposition.limit}
If the ROBM pre-homology theory $B_{*}$ on $\Schk$ satisfies properties $\operatorname{(H)}$, $\operatorname{(L)}$ and $\operatorname{(D)}$, then for any $X$ in $\Schkg$ there exists an isomorphism
\[
B^{*}_{G}(X) \isom \varprojlim_{i} B^{*}(X \times^{G} U_{i}).
\]
The isomorphism is natural with respect to pull-back by any morphism $f: X'\to X$ in $\Schkg$.
\end{proposition}
\begin{proof}
Given $X$ in $\Schkg$ we define a homomorphism 
\[
\varphi_{X}: \varprojlim_{i} B^{*}(X \times^{G} U_{i}) \ra B^{*}_{G}(X)
\]
as follows: Given $c=(c_{i}) \in \varprojlim_{i} B^{*}(X \times^{G} U_{i})$, where $c_{i} \in B^{*}(X \times^{G} U_{i})$ for each $i$, and given any $G$-equivariant morphism $f: Y \ra X$ and a class $\alpha=(\alpha_{i}) \in B_{*}^{G}(Y)=\varprojlim_{i}B_{*}(Y \times^{G} U_{i})$, where $\alpha_{i} \in B_{*}(Y \times^{G} U_{i})$ for each $i$, one defines $\varphi_{X}(c)(\alpha)=(c_{i}(\alpha_{i})) \in \varprojlim_{i}B_{*}(Y \times^{G} U_{i}) = B_{*}^{G}(Y)$. It is straightforward to verify that $\varphi_{X}(c)$ is well defined and is indeed a bivariant class in $B^{G}_{*}(X)$, so that $\varphi_{X}$ is a well defined homomorphism. 

It is clear from the definitions that $\varphi_{X}$ is natural with respect to pull-backs. 

We prove that $\varphi_{X}$ is an isomorphism. First, we consider the case when $X$ is smooth. In this case, each $X \times^{G} U_{i}$ is smooth as well. By Poincar\'e duality (Proposition~\ref{poincare}) we have isomorphisms $B^{*}_{G}(X) \xrightarrow{\cap 1_{X}} B_{*}^{G}(X)$ and $B^{*}(X \times^{G} U_{i}) \xrightarrow{\cap 1_{X \times^{G} U_{i}}} B_{*}(X \times^{G} U_{i})$ for each $i$. Passing to the component-wise inverse limit, and composing appropriately, one obtains the isomorphism
\[\xymatrixcolsep{5pc}
\xymatrix{
{\varprojlim_{i} B^{*}(X \times^{G} U_{i})}      \ar[r]^-{(\cap 1_{X \times^{G} U_{i}})_{i}}   \ar@/_2pc/[rr]^{\varphi_{X}}              &
{\varprojlim_{i} B_{*}(X \times^{G} U_{i})}      =       {B_{*}^{G}(X)}                                      \ar[r]^-{(\cap 1_{X})^{-1}}    &
{B^{*}_{G}(X)}
}
\]
which can be verified directly to be $\varphi_{X}$.

For the general case, given $X$ we choose a $G$-equivariant envelope $\pi:\tilde{X}\ra X$ so that $\pi$ is projective and $\tilde{X}$ is smooth. Let $\pi_{i}:\tilde{X} \times^{G} U_{i} \ra X \times^{G} U_{i}$ be the induced morphisms. We get a commutative diagram:
\[\xymatrixcolsep{4pc}
\xymatrix{
{\varprojlim_{i} B^{*}(X \times^{G} U_{i})} \ar[d]^{\varphi_{X}} \ar@{^{(}->}[r]^{(\pi_{i}^{*})} &{\varprojlim_{i} B^{*}(\tilde{X} \times^{G} U_{i})} \ar[d]_{\wr}^{\varphi_{\tilde{X}}} \\
{B^{*}_{G}(X)} \ar@{^{(}->}[r]^{\pi^{*}} &  {B^{*}_{G}(\tilde{X})} }
\]
We claim that $\pi^{*}$ and $(\pi_{i}^{*})$ are injective. To see this, assume that $\pi^*c=0$ for some $c\in B^{*}_{G}(X)$. Given a $G$-equivariant map $f:Y\ra X$ and a class $\alpha \in B_{*}^{G}(Y)$, form the fiber product
\[
\xymatrix{
\tilde{Y} \ar[d]^{f'} \ar[r]^{\pi'} &Y \ar[d]^f \\
\tilde{X} \ar[r]^{\pi}&X }
\]
with morphisms as indicated. Since $\pi'$ is an envelope, by Lemma \ref{surjectivity.equivariant.envelopes} there exists $\tilde{\alpha} \in B_{*}^{G}(\tilde{Y})$ such that $\pi'_{*}(\tilde{\alpha})=\alpha$. We have $c(\alpha)=c(\pi'_{*}(\tilde{\alpha}))=\pi'_{*}(c(\tilde{\alpha}))=\pi'_{*}((\pi^{*}c)(\tilde{\alpha}))=0$, and it follows that $\pi^{*}$ is injective. Since each $\pi_{i}$ is also an envelope, the same argument proves that the $\pi_{i}^{*}$ are injective, and then so is the component-wise inverse limit homomorphism $(\pi_{i}^{*})$. In particular, it follows that $\varphi_{X}$ is injective, as $\varphi_{\tilde{X}}$ is an isomorphism by the smooth case. 

To prove the surjectivity of $\varphi_{X}$, we consider $c\in {B^{*}_{G}(X)}$, and construct an element mapping to $\varphi^{-1}_{\tilde{X}} (\pi^{*}c)$ by $(\pi_{i}^{*})$. Let $p_{j}:\tilde{X}\times_{X}\tilde{X}\ra \tilde{X}$ and $p'_{j}: (\tilde{X} \times^{G} U_{i})   \times_{X\times^{G} U_{i}} (\tilde{X}\times^{G} U_{i}) \ra \tilde{X} \times^{G} U_{i}$ be the projections on the corresponding $j^{\textnormal{th}}$ factor, for $j=1,2$. Let $(\tilde{c}_{i}) \in \varprojlim_{i} B^{*}(\tilde{X} \times^{G} U_{i})$ be the image of $c$ under $\varphi^{-1}_{\tilde{X}} \circ \pi^*$, where each $\tilde{c}_{i} \in B^{*}(\tilde{X} \times^{G} U_{i})$. Now, using that $\varphi_{\tilde{X} \times_{X}  \tilde{X}}$ is injective and that $(p_{1}^{*}-p_{2}^{*})(\pi^{*}c)=0$, it follows that $({p'_{1}}^{*}-{p'_{2}}^{*})(\tilde{c}_{i})=0$ for each $i$. We define a class $(c_{i}) \in \varprojlim_{i} B^{*}(X \times^{G} U_{i})$ as follows: Given a morphism $f:Y \ra X \times^{G} U_{i}$ and a class $\alpha \in B_{*}(Y)$, form the fiber product
\[
\xymatrix{
\tilde{Y} \ar[d]^{f'} \ar[r]^{\pi'_{i}} &Y \ar[d]^f \\
\tilde{X}\times^{G} U_{i} \ar[r]^{\pi_{i}}&X\times^{G} U_{i} }
\]
with morphisms as indicated, and define $c_{i}(\alpha)=\pi'_{i \, *}(\tilde{c}_{i}(\tilde{\alpha}))$, where $\tilde{\alpha}\in B_{*}(\tilde{Y})$ is any class satisfying $\pi'_{i \, *}(\tilde{\alpha})=\alpha$ (which exists since $\pi'_{i}$ is an envelope). To see that $c_{i}(\alpha)$ is independent of the choice of $\tilde{\alpha}$, it is enough to see that $\pi'_{i \, *}(\tilde{c}_{i}(\beta))=0$ for each class $\beta \in B_{*}(\tilde{Y})$ such that $\pi'_{i \, *}\beta=0$. By property $\operatorname{(D)}$, given such class $\beta$ there exists a class $\gamma \in B_{*}(\tilde{Y} \times_{Y} \tilde{Y})$ such that $\beta=g_{1*}(\gamma)-g_{2*}(\gamma)$, where $g_{j}:\tilde{Y} \times_{Y} \tilde{Y} \ra \tilde{Y}$ are the projections, for $j=1,2$. Since $\pi'_{i}\circ g_{1}=\pi'_{i}\circ g_{2}$, it follows that $\pi'_{i \, *}(\tilde{c}_{i}(\beta))=\pi'_{i \, *}(\tilde{c}_{i}(g_{1*}(\gamma)-g_{2*}(\gamma)))=((\pi'_{i}\circ g_{1})_{*}-(\pi'_{i}\circ g_{2})_{*})(\tilde{c}_{i}(\gamma))=0$, and then $c_{i}$ is well defined. It is straightforward to verify 
that each $c_{i}$ satisfies the conditions $\operatorname{(C_{1})}$-$\operatorname{(C_{4})}$, so they define bivariant classes $c_{i}\in B^{*}(X\times^{G} U_{i})$. Moreover, it is clear that the classes $c_{i}$ agree under the bivariant pull-backs $B^{*}(X\times^{G} U_{j}) \ra B^{*}(X\times^{G} U_{i})$ for each $j \geq i$, so they define a class $(c_{i}) \in \varprojlim_{i} B^{*}(X \times^{G} U_{i})$. To prove that $\varphi_{X}$ is surjective, it is enough to see that $\pi_{i}^{*}c_{i}=\tilde{c}_{i}$ for each $i$. For this, let $f:Y \ra \tilde{X}\times^{G} U_{i}$ be any morphism and form the fiber diagram 
\[\xymatrixcolsep{5pc}
\xymatrix{
\tilde{Y} \ar[d]^{f'} \ar[r]^{\pi'_{i}} &Y \ar[d]^f \\
(\tilde{X} \times^{G} U_{i})   \times_{X\times^{G} U_{i}} (\tilde{X}\times^{G} U_{i}) \ar[d]^{p'_1} \ar[r]^-{p'_2} & \tilde{X}\times^{G} U_{i} \ar[d]^{\pi_{i}} \\
\tilde{X}\times^{G} U_{i} \ar[r]^{\pi_{i}}& X\times^{G} U_{i} }
\]
with morphisms as indicated. Consider a class $\alpha \in B_{*}(Y)$ and any class $\tilde{\alpha} \in B_{*}(\tilde{Y})$ such that $\pi'_{i \, *}(\tilde{\alpha})=\alpha$. Since ${p}_1'^{*}(\tilde{c}_{i})={p}_2'^{*}(\tilde{c}_{i})$, we have that
\begin{align*}
	(\pi_{i}^{*}c_{i})_{f}(\alpha)&=\pi'_{i \, *}(       (\tilde{c}_{i})_{p'_{1}\circ f'} (\tilde{\alpha})  )
                              = \pi'_{i \, *}(  ( f'^{*} {p}_1'^{*} \tilde{c}_{i}) (\tilde{\alpha})  )       \\
                              &=\pi'_{i \, *}(  ( f'^{*} {p}_2'^{*} \tilde{c}_{i}) (\tilde{\alpha})  )
                              =\pi'_{i \, *}(       (\tilde{c}_{i})_{f \circ \pi'_{i} } (\tilde{\alpha})  )
                              =(\tilde{c}_{i})_{f}(\alpha). 
\end{align*}
Hence $\pi_{i}^{*}c_{i}=\tilde{c}_{i}$ for each $i$, and the proof is complete.
\end{proof}


\subsection{Kimura Type Descent Sequence for Bivariant Theories and Inductive Computation of Bivariant Groups}
 
Theorem~\ref{thm-seq2-biv} and Theorem~\ref{thm-biv-image} below were proved by Kimura \cite{Kimura} in the case of Chow theory $A_*$. We generalize his proofs to arbitrary ROBM pre-homology theories.

\begin{theorem}  \label{thm-seq2-biv} Let $B_{*}$ be an ROBM pre-homology theory on $\mathcal{C}=\Schk$ or $\mathcal{C}=\Schkg$ that satisfies property $\operatorname{(D)}$. 
Let $\pi: \tilde{X}\to X$ be a projective envelope in $\mathcal{C}$, $Y \to X$ be a morphism in $\mathcal{C}$, and $\tilde{Y} = \tilde{X}\times_X Y$. Then the following sequence is exact
\begin{equation} \label{sequence.kimura.bivariant} 0 \xrightarrow{} B^{*}(Y\to X) \xrightarrow{\pi^*}  B^{*}(\tilde{Y}\to \tilde{X}) \xrightarrow{p_1^*-p_2^*} B^{*}( \tilde{Y}\times_Y\tilde{Y} \to \tilde{X}\times_X\tilde{X}). \end{equation}  
\end{theorem}
\begin{proof}
Assume that $\pi^*c=0$ for some $c\in B^{*}(Y \ra X)$. Given a morphism $f: X' \ra X$ and a class $\alpha \in B_{*}(X')$, form the fiber diagram on the left in (\ref{fiber.diagram.cube})
\begin{equation} \label{fiber.diagram.cube}
\begin{minipage}[b]{0.5\linewidth}
\SelectTips{eu}{12}
\[
\xymatrix{ 
& \tilde{Y'} \ar'[d][dd] \ar[ld]_{\pi''} \ar[rr]  &&\tilde{X'} \ar[dd] \ar[ld]^{\pi'} \\
Y'  \ar[dd]  \ar[rr]  &&  X'  \ar[dd]  \\
& \tilde{Y}   \ar'[r][rr]  \ar[ld]   &&  \tilde{X}  \ar[ld]^{\pi}   \\
Y \ar[rr]  && X   }
\]
\end{minipage}
\begin{minipage}[b]{0.5\linewidth}
%
\SelectTips{eu}{12}
\[\xymatrixcolsep{6mm}\xymatrixrowsep{6mm}
\xymatrix{ 
& \tilde{Y'} \ar'[d][dd] \ar[ld]_{\pi''} \ar[rr]  &&\tilde{X'} \ar[dd]^{f'} \ar[ld]^{\pi'} \\
Y'  \ar[dd]  \ar[rr]  &&  X'  \ar[dd]^<<<<<{f}  \\
& \tilde{Y}\times_{Y}\tilde{Y} \ar'[d][dd] \ar[ld] \ar'[r][rr]  &&\tilde{X}\times_{X}\tilde{X} \ar[dd]^{p_1} \ar[ld]^{p_2} \\
\tilde{Y}  \ar[dd]  \ar[rr]  &&  \tilde{X}  \ar[dd]^<<<<<{\pi}  \\
& \tilde{Y}   \ar'[r][rr]  \ar[ld]   &&  \tilde{X}  \ar[ld]^{\pi}   \\
Y \ar[rr]  && X   }
\]
%
\end{minipage}
\end{equation} 
with morphisms labeled as indicated. Since $\pi'$ is an envelope, there exists $\tilde{\alpha} \in B_{*}(\tilde{X'})$ such that $\pi'_{*}(\tilde{\alpha})=\alpha$. We have $c(\alpha)=c(\pi'_{*}(\tilde{\alpha}))= \pi''_{*} (c(\tilde{\alpha}))= \pi''_{*} ((\pi^{*}c)(\tilde{\alpha}))=0$, and it follows that $\pi^{*}$ is injective. 
By the functoriality of pull-backs it follows that $(p_1^*-p_2^*) \circ \pi^* =0$. 

Now, let $\tilde{c} \in B^{*}(\tilde{Y} \ra \tilde{X})$ be a bivariant class such that $(p_1^*-p_2^*)(\tilde{c})=0$.
We define a class $c \in B^{*}(Y \ra X)$ as follows: Given a morphism $f:X' \ra X$ and a class $\alpha \in B_{*}(X')$, form once again the fiber diagram on the left in (\ref{fiber.diagram.cube}), with morphisms as indicated, and define $c(\alpha)= \pi''_{*} (\tilde{c}(\tilde{\alpha}))$, where $\tilde{\alpha}\in B_{*}(\tilde{X'} )$ is any class satisfying $\pi'_{*}(\tilde{\alpha})=\alpha$ (which exists since $\pi'$ is an envelope). To see that $c(\alpha)$ is independent of the choice of $\tilde{\alpha}$, it is enough to see that $ \pi''_{*} (\tilde{c}(\beta))=0$ for each class $\beta \in B_{*}(\tilde{X'})$ such that $\pi'_{*}\beta=0$. Let $g_{j}:\tilde{X'} \times_{X'} \tilde{X'} \ra \tilde{X'}$ and $g'_{j}:\tilde{Y'} \times_{Y'} \tilde{Y'} \ra \tilde{Y'}$ be the projections, for $j=1,2$. By property $\operatorname{(D)}$, given such class $\beta$ there exists a class $\gamma \in B_{*}(\tilde{X'} \times_{X'} \tilde{X'})$ such that $\beta=g_{1*}(\gamma)-g_{2*}(\gamma)$. Since $\pi'' \circ g'_{1} = \pi'' \circ g'_{2}$, it follows that $ \pi''_{*} (\tilde{c}(\beta)) = \pi''_{*} (\tilde{c}(g_{1*}(\
\gamma)-g_{2*}(\gamma)))=(( \pi'' \circ g'_{1})_{*}-( \pi'' \circ g'_{2})_{*})(\tilde{c}(\gamma))=0$, and then $c$ is well defined. It is straightforward to verify that $c$ satisfies the conditions $\operatorname{(C_{1})}$-$\operatorname{(C_{4})}$, so this construction yields a bivariant class $c \in B_{*}( Y \ra X)$. To finish the proof we show that $\pi^{*}c=\tilde{c}$. For this, let $f: X' \ra \tilde{X} $ be any morphism and consider a class $\alpha \in B_{*}(X')$ and any class $\tilde{\alpha} \in B_{*}(\tilde{X'})$ such that $\pi'_{*}(\tilde{\alpha})=\alpha$. Form the fiber diagram on the right in (\ref{fiber.diagram.cube}) with morphisms as indicated. Since $p_{1}^{*}(\tilde{c})=p_{2}^{*}(\tilde{c})$, we have that
\begin{align*}
	(\pi^{*}c)_{f}(\alpha)&= \pi''_{*} (   \tilde{c}_{p_{1}\circ f'} (\tilde{\alpha})  )
                              = \pi''_{*}(  ( f'^{*} p_{1}^{*} \tilde{c}) (\tilde{\alpha})  )       \\
                              &= \pi''_{*} (  ( f'^{*} p_{2}^{*} \tilde{c}) (\tilde{\alpha})  )
                              = \pi''_{*} (    \tilde{c}_{f \circ \pi' } (\tilde{\alpha})  )
                              =\tilde{c}_{f}(\alpha). 
\end{align*}
Hence $\pi^{*}c=\tilde{c}$, and the proof is complete.
\end{proof}

\begin{corollary}  \label{corollary.operational.kimura}
Let $B_{*}$ be an ROBM pre-homology theory on $\mathcal{C}=\Schk$ or $\mathcal{C}=\Schkg$ that satisfies property $\operatorname{(D)}$. Then for any projective envelope $\pi: \tilde{X}\to X$ in $\mathcal{C}$ the following sequence is exact
\[ 0 \xrightarrow{} B^{*}(X) \xrightarrow{\pi^*}  B^{*}(\tilde{X}) \xrightarrow{ p_1^*- p_2^*} B^{*}( \tilde{X}\times_X\tilde{X}). \]  
\end{corollary}

The following lemma is a simple consequence of the localization property $\operatorname{(L)}$ or of the property $\operatorname{(D)}$.

\begin{lemma}  \label{classes.decomposition}
Let $B_{*}$ be an ROBM pre-homology theory on $\mathcal{C}=\Schk$ or $\mathcal{C}=\Schkg$ that satisfies either property $\operatorname{(L)}$ or property $\operatorname{(D)}$. If $X= \bigcup_{i=1}^{r}Z_{i}$ where each $f_{i}:Z_{i} \ra X$ is a closed subscheme of $X$ then
\[
B_{*}(X)=\Sigma_{i=1}^{r} f_{i \, *} (B_{*}(Z_{i})) 
\] 
\end{lemma}
\begin{proof}
The general case follows at once from the case $r=2$, so we assume that $X= Z_{1} \cup Z_{2}$. If $B_{*}$ satisfies $\operatorname{(L)}$, we have that the following sequence is exact
\[
B_{*}(Z_{1}) \xrightarrow{ f_{1 \, *} } B_{*}(X) \xrightarrow{f_{2}|^{*}} B_{*}(X \smallsetminus Z_{1}) \xrightarrow{} 0.
\] 
By the compatibility of pull-backs and push-forwards and using localization it is clear that $f_{2}|^{*}$ maps $f_{2 \, *}(B_{*}(Z_{2}))$ onto $B_{*}(X \smallsetminus Z_{1})$. It is clear now that $B_{*}(X)= f_{1 \, *} (B_{*}(Z_{1})) + f_{2 \, *}(B_{*}(Z_{2}))$ as desired. If $B_{*}$ satisfies $\operatorname{(D)}$, the result follows since $B_{*}$ is additive and the projective morphism from the disjoint union $Z_{1} \coprod Z_{2} \ra X$ induced by the inclusions is an envelope.
\end{proof}

The following result can be proved as a corollary of Theorem~\ref{thm-seq2-biv}. It gives an inductive method for computing bivariant groups. 

\begin{theorem} \label{thm-biv-image}
Let $B_{*}$ be an ROBM pre-homology theory on $\mathcal{C}=\Schk$ or $\mathcal{C}=\Schkg$.
Let $\pi: \tilde{X} \to X$ be a projective and birational envelope in $\cC$. Let $Y \to X$ be a morphism in $\cC$ and $\tilde{Y} = \tilde{X}\times_X Y$.
Assume that $B_*$ satisfies the conclusion of Lemma~\ref{classes.decomposition} and that the sequence $(\ref{sequence.kimura.bivariant})$ in Theorem~\ref{thm-seq2-biv} is exact (e.g. it is enough to assume that $B_{*}$ satisfies $\operatorname{(D)}$).
Let $\pi: \pi^{-1}(U)\isomto U$ for some open dense $U\subset X$. Let $S_i\subset X$ be closed subschemes, such that $X\setmin U = \cup S_i$. Let $E_i = \pi^{-1}(S_i)$ and let $\pi_i: E_i\to S_i$ be the induced morphism. Then for a class $\tilde{c}\in B^{*}(\tilde{Y}\to\tilde{X})$ the following are equivalent:
\begin{enumerate}
 \item \label{kim.i} $\tilde{c} = \pi^*(c)$ for some $c\in B^{*}(Y\to X)$.
\item \label{kim.ii} For all $i$, $\tilde{c}|_{E_i} = \pi_i^*(c_i)$ for some $c_i\in B^{*}(Y \times_X S_i \to S_i)$.
\end{enumerate}
\end{theorem}
\begin{proof}
If $\tilde{c} = \pi^*(c)$ for some $c\in B^{*}(Y\to X)$, then by the functoriality of pull-backs $\tilde{c}|_{E_i} = \pi_i^*(c|_{S_i})$ for all $i$, and then (\ref{kim.ii}) holds if we take $c_{i} = c|_{S_i} \in B^{*}(Y \times_X S_i \to S_i)$. Reciprocally, assume that there are classes $c_{i}$ as in (\ref{kim.ii}). Let $p_{1},p_{2}:\tilde{X}\times_{X}\tilde{X}\ra \tilde{X}$ be the projections. By Theorem~\ref{thm-seq2-biv} it is enough to show that $p_{1}^{*}\tilde{c}=p_{2}^{*}\tilde{c}$, i.e., that for any morphism $f: Z \ra \tilde{X} \times_{X} \tilde{X}$ and for any class $\alpha \in B_{*}Z$ we have that $(p_{1}^{*}\tilde{c})(\alpha)=(p_{2}^{*}\tilde{c})(\alpha)$. Let $f'_{i}: E_{i}\times_{X}E_{i} \ra \tilde{X} \times_{X} \tilde{X}$ be the corresponding closed embeddings and let $\Delta: \tilde{X} \ra \tilde{X} \times_{X} \tilde{X}$ be the diagonal morphism which is also a closed embedding. Notice that $\tilde{X} \times_{X} \tilde{X}$ is the union the closed subschemes $\Delta(\tilde{X})$ and $f'_{i}(E_{i} \times_{X} E_{i})$ for all $i$. Let $Z_{0}=f^{-1}(\Delta(\tilde{X}))$ and $Z_{i}=f^{-1}(f_{i}'(E_{i} \times_{X} E_{i}))$, with inclusions $\Delta':Z_{0} \ra Z$ and ${f''_{i}}:Z_{i} \ra Z$, for each $i$. 
Then, by Lemma~\ref{classes.decomposition}, in order to prove that $(p_{1}^{*}\tilde{c})(\alpha)=(p_{2}^{*}\tilde{c})(\alpha)$, we can assume that either $\alpha= f''_{i \, *} (\alpha_{i})$ for some $i$ and some $\alpha_{i} \in B_{*}(Z_{i})$ or $\alpha= \Delta'_* (\alpha_{0})$ for some $\alpha_{0} \in B_{*}(Z_{0})$. 
In the first case, for $j=1$ and $j=2$ consider the fiber diagram
\[
\xymatrix{
E_{i}\times_{X}E_{i} \ar[d]^{p_{j}|} \ar[r]^{f'_{i}} &\tilde{X} \times_{X} \tilde{X} \ar[d]^{p_{j}} \\
E_{i} \ar[d]^{\pi_{i}} \ar[r]^{\tilde{f_{i}}} &\tilde{X} \ar[d]^{\pi} \\
S_{i} \ar[r]^{f_{i}} &X }
\]
with morphisms as labeled. Let $\Delta'':Z_{0} \times_{X} Y \ra Z \times_{X} Y$ and $g''_{i}:Z_{i} \times_{X} Y \ra Z \times_{X} Y$ be the morphisms obtained from $\Delta'$ and $f''_{i}$ by base change. We have
\begin{align*}
(p_{j}^{*}\tilde{c})(\alpha)
&=(p_{j}^{*}\tilde{c})( f''_{i \, *} (\alpha_{i}))= g''_{i \, *} ((p_{j}^{*}\tilde{c})(\alpha_{i}))= g''_{i \, *} (\tilde{c}(\alpha_{i}))  \\
&= g''_{i \, *} ((\tilde{f_{i}}^{*}\tilde{c})(\alpha_{i}))= g''_{i \, *} (({\pi_{i}}^{*}c_{i})(\alpha_{i}))= g''_{i \, *} (({p_{j}|}^{*}{\pi_{i}}^{*}c_{i})(\alpha_{i}))    \\
&= g''_{i \, *} ((({\pi_{i}} \circ p_{j}| )^{*}c_{i})(\alpha_{i})).	
\end{align*}
Since ${\pi_{i}} \circ p_{1}| ={\pi_{i}} \circ p_{2}|$, it follows that in the first case $(p_{1}^{*}\tilde{c})(\alpha)=(p_{2}^{*}\tilde{c})(\alpha)$. In the second case, for $j=1$ and $j=2$, we have
\begin{align*}
(p_{j}^{*}\tilde{c})(\alpha)
&=(p_{j}^{*}\tilde{c})(\Delta'_{*}(\alpha_{0}))=\Delta''_{*}((p_{j}^{*}\tilde{c})(\alpha_{0}))=\Delta''_{*}((\Delta^{*}p_{j}^{*}\tilde{c})(\alpha_{0}))   \\
&=\Delta''_{*}(((p_{j} \circ \Delta)^{*}\tilde{c})(\alpha_{0})) = \Delta''_{*}(({\id_{\tilde{X}}}^{*}\tilde{c})(\alpha_{0})) = \Delta''_{*}(\tilde{c}(\alpha_{0})).
\end{align*}
Therefore, in the second case $(p_{1}^{*}\tilde{c})(\alpha)=(p_{2}^{*}\tilde{c})(\alpha)$, and the proof is complete.
\end{proof}


\subsection{Kimura Type Descent Sequence for Bivariant Equivariant Theories and Inductive Computation of Bivariant Equivariant Groups}    \label{kimura.induced.equivariant.OBM}

Theorem~\ref{thm-seq2-biv} and Theorem~\ref{thm-biv-image} proved above can be applied to the equivariant theory $B_*^G$, provided that it satisfies property $\operatorname{(D)}$. We can not prove property $\operatorname{(D)}$ for $B_*^G$ assuming that it holds for $B_*$. We will therefore give a different proof of the statements of Theorem~\ref{thm-seq2-biv} and Theorem~\ref{thm-biv-image} for $B_*^G$ that depends only on $B_*$ satisfying property $\operatorname{(D)}$. We give proofs for the bivariant cohomology groups $B^*_G(X)$ only.

\begin{theorem} \label{theorems.bivariant.equivariant} Let $B_{*}$ be an ROBM pre-homology theory on $\Schk$ that satisfies properties $\operatorname{(H)}$, $\operatorname{(L)}$ and $\operatorname{(D)}$. Let $\pi:\tilde{X} \ra X$ be a projective envelope in $\Schkg$, and let the terminology be as in Theorem~\ref{thm-biv-image}.
\begin{itemize}
\item[$\operatorname{(a)}$] The following sequence is exact
\[ 0 \xrightarrow{} B_{G}^{*}(X) \xrightarrow{\pi^*}  B_{G}^{*}(\tilde{X}) \xrightarrow{p_1^*-p_2^*} B_{G}^{*}( \tilde{X}\times_X\tilde{X}). \] 
\item[$\operatorname{(b)}$] If $\pi$ is also birational, then for a class $\tilde{c}\in B_{G}^{*}(\tilde{X})$ the following are equivalent:
\begin{enumerate}
 \item \label{kim.i.2} $\tilde{c} = \pi^*(c)$ for some $c\in B_{G}^{*}(X)$.
\item \label{kim.ii.2} For all $i$, $\tilde{c}|_{E_i} = \pi_i^*(c_i)$ for some $c_i\in B_{G}^{*}(S_i)$.
\end{enumerate}
\end{itemize}
\end{theorem}
\begin{proof}
\begin{itemize}
\item[$\operatorname{(a)}$] Since for each $i$ the map $\tilde{X}\times^G U_i \to X\times^G U_i$ is an envelope,  by Corollary \ref{corollary.operational.kimura} the sequence
 \[ 0 \longrightarrow B^*(X \times^{G} U_{i} ) {\longrightarrow} B^*(\tilde{X} \times^{G} U_{i}) {\longrightarrow} B^*( (\tilde{X}\times_X\tilde{X}) \times^{G}   U_{i} )
 \]
is exact. Applying the left exact functor $\varprojlim$ and using Proposition~\ref{proposition.limit} gives the desired result.

\item[$\operatorname{(b)}$] In view of part $\operatorname{(a)}$, the conclusion follows from Theorem~\ref{thm-biv-image} if we show that the ROBM pre-homology theory $B^{G}_{*}$ satisfies the conclusion of Lemma~\ref{classes.decomposition}. 

For this, it is enough to consider the case $r=2$, so we let $X=Z_{1} \cup Z_{2}$ and Let $Z = Z_{1}\coprod Z_{2}$. The projective morphism $Z \ra X$ induced by the inclusions is an envelope, hence $B_*^G(Z)\to B_*^G(X)$ is surjective by Lemma~\ref{surjectivity.equivariant.envelopes}  \qedhere
\end{itemize} 
\end{proof}

\section{An Overview of Algebraic Cobordism Theory}              \label{section.overview.cobordism}

In this section we recall the definition and main properties of \emph{algebraic cobordism} $\Omega_{*}$. This theory was constructed by Levine and Morel in \cite{Levine-Morel}. Later, Levine and Pandharipande \cite{Levine-Pandharipande} found a geometric presentation of the cobordism groups. We will use the construction of Levine-Pandharipande as the definition, but refer to Levine-Morel for its properties. This construction and the proofs of some of facts stated below use resolution of singularities, factorization of birational maps and some Bertini-type theorems, then we will assume for the remainder of this article that the field $k$ has characteristic zero.

The \emph{equivariant algebraic cobordism} $\Omega^G_*$ was constructed first by Krishna \cite{Krishna} and by Heller and Malag\'on-L\'opez \cite{Heller-Malagon}. Krishna and Uma \cite{Krishna-Uma} showed how to compute the equivariant and ordinary cobordism groups of smooth toric varieties; we will recall their result in \S\,\ref{section.cobordism.toric}.

For $X$ in $\Schk$, let $\cM(X)$ be the set of isomorphism classes of projective morphisms $f: Y\to X$ for $Y\in \Smk$. This set is a monoid under disjoint union of the domains; let $\Mp(X)$ be its group completion. The elements of $\Mp(X)$ are called cycles. The class of $f: Y\to X$ in $\Mp(X)$ is denoted $[f: Y\to X]$. The group $\Mp(X)$ is free abelian, generated by the cycles $[f: Y\to X]$ where $Y$ is irreducible.

A double point degeneration is a morphism $\pi: Y\to \PP^1$, with $Y \in \Smk$ of pure dimension, such that $Y_\infty = \pi^{-1}(\infty)$ is a smooth divisor on $Y$ and $Y_0=\pi^{-1}(0)$ is a union $A\cup B$ of smooth divisors intersecting transversely along $D=A\cap B$.  Define $$\PP_D = \PP(\cO_D(A)\oplus \cO_D) = \Proj_{\mathcal{O}_{D}}(\Sym_{\mathcal{O}_{D}}^{*}(\cO_D(A)\oplus \cO_D)),$$ where $\cO_D(A)$ stands for $\cO_Y(A)|_D$. (Notice that $\PP(\cO_D(A)\oplus \cO_D) \isom \PP(\cO_D(B) \oplus \cO_D)$ because $\cO_D(A+B)\isom \cO_D$.) 

Let $X\in \Schk$ and let $Y\in \Smk$ have pure dimension. Let $p_1, p_2$ be the two projections of $X\times \PP^1$.  A double point relation is defined by a projective morphism $\pi: Y\to X\times\PP^1$, such that $p_2\circ \pi: Y\to \PP^1$ is a double point degeneration. Let 
\[ [Y_\infty \to X], \quad [A\to X],\quad [B\to X], \quad [\PP_D \to X] \]
be the cycles obtained by composing with $p_1$. The double point relation is 
\[ [Y_\infty \to X] -[A\to X] - [B\to X] + [\PP_D\to X] \in \Mp(X). \]

Let $\Rel(X)$ be the subgroup of $\Mp(X)$ generated by all the double point relations. The cobordism group of $X$ is defined to be
\[ \Omega_*(X) = \Mp(X)/\Rel(X).\]
The group $\Mp(X)$ is graded so that $[f: Y\to X]$ lies in degree $\dim Y$ when $Y$ has pure dimension. Since double point relations are homogeneous, this grading gives a grading on $\Omega_*(X)$. We write $\Omega_n(X)$ for the degree $n$ part of  $\Omega_*(X)$. 

There is a functorial push-forward homomorphism $f_*: \Omega_*(X)\to \Omega_*(Z)$ for $f: X\to Z$ projective, and a functorial pull-back homomorphism $g^*: \Omega_*(Z)\to \Omega_{*+d}(X)$ for $g: X\to Z$ a  
smooth morphism of relative dimension $d$. These homomorphisms are both defined on the cycle level; the pull-back does not preserve grading. The exterior product on $\Omega_*(X)$ is also defined on the cycle level:
\[ [Y\to X] \times [Z\to W] = [Y\times Z \to X\times W].\]
A much harder result proved in \cite{Levine-Morel} is the existence of  functorial pull-backs $g^*$ along l.c.i. morphisms $g$, and more generally, the existence of refined l.c.i. pull-backs.

The groups $\Omega_*(X)$ with these projective push-forward, l.c.i. pull-back and exterior products form an oriented Borel-Moore homology theory (see \cite{Levine-Morel}). Moreover, with those refined l.c.i. pull-backs it is also an ROBM pre-homology theory (see \cite{Levine-Morel}). 

As in the case of a general ROBM pre-homology theory, $\Omega_*(\Spec k)$ is a ring, $\Omega_*(X)$ is a module over $\Omega_*(\Spec k)$ for general $X$, and $\Omega_*(X)$ is an algebra over $\Omega_*(\Spec k)$ for smooth $X$. When $X$ is smooth and has pure dimension, we also use the cohomological notation
\[ \Omega^*(X) = \Omega_{\dim X- *}(X).\]
Then $\Omega^*(X)$ is a graded algebra over the graded ring $\Omega^*(\Spec k)$. The class $1_X = [\id_X: X\to X]$ is the identity of the algebra. Similar conventions are used for the equivariant cobordism groups.

\begin{remark}
Algebraic cobordism satisfies the homotopy property $\operatorname{(H)}$ and the localization property $\operatorname{(L)}$ \cite{Levine-Morel}, as well as the descent property $\operatorname{(D)}$ \cite{descentseq}. It follows that \emph{everything} proved in the previous sections for general ROBM pre-homology theories can be applied to the algebraic cobordism theory. This includes the construction of the equivariant cobordism theory $\Omega^G_*$, the operational cobordism theory $\Omega^*$ and the operational equivariant cobordism theory $\Omega_G^*$. This also includes the descent exact sequences for the operational theories $\Omega^*$ and $\Omega_G^*$ and the inductive method for their computation using envelopes. 
\end{remark}

The following part of Theorem~\ref{theorems.bivariant.equivariant} applied to $\Omega_G^*(X)$ will be used in the next section.

\begin{theorem} \label{thm-ind}
Assume that $\pi: \tilde{X}\to X$ is a projective birational envelope in $\Schkg$, with $\pi: \pi^{-1}(U)\isomto U$ for some open nonempty $G$-equivariant $U\subset X$. Let $S_i\subset X$ be closed $G$-equivariant subschemes, such that $X\setmin U = \cup S_i$. Let $E_i = \pi^{-1}(S_i)$ and let $\pi_i: E_i\to S_i$ be the induced morphism. Then $\pi^*: \Omega_G^*(X) \to \Omega^*_G(\tilde{X})$ is injective and for a class $\tilde{c}\in \Omega^*_G(\tilde{X})$ the following are equivalent:
\begin{enumerate}
 \item $\tilde{c} = \pi^*(c)$ for some $c\in\Omega^*_G(X)$.
\item For all $i$, $\tilde{c}|_{E_i} = \pi_i^*(c_i)$ for some $c_i\in \Omega^*_G(S_i)$.
\end{enumerate}
\end{theorem}

\subsection{Formal group law}

Algebraic cobordism is endowed with first Chern class operators associated to line bundles, whose definition agrees with the one in Definition \ref{definition.obm.first.chern.class}. We recall the formal group law satisfied by these operators.

A \emph{formal group law} on a commutative ring $R$ is a power series $F_R(u,v)\in R\llbracket u,v\rrbracket $ satisfying
\begin{enumerate}[(a)]
\item $F_R(u,0) = F_R(0,u) = u$,
\item $F_R(u,v)=  F_R(v,u)$,
\item $F_R(F_R(u,v),w) = F_R(u,F_R(v,w))$.
\end{enumerate}
Thus 
\[ F_R(u,v) = u+v +\sum_{i,j>0} a_{i,j} u^i v^j,\]
where $a_{i,j}\in R$ satisfy $a_{i,j}=a_{j,i}$ and some additional relations coming from property (c). We think of $F_R$ as giving a formal addition
\[ u+_{F_R} v = F_R(u,v).\]
There exists a unique power series $\chi(u) \in R\llbracket u\rrbracket $ such that $F_R(u,\chi(u)) = 0$. Denote $[-1]_{F_R} u = \chi(u)$. Composing $F_R$ and $\chi$, we can form linear combinations 
\[ [n_1]_{F_R} u_1 +_{F_R}   [n_2]_{F_R} u_2 +_{F_R} \cdots +_{F_R} [n_r]_{F_R} u_r \in R\llbracket u_1,\ldots,u_r\rrbracket \]
for $n_i\in \ZZ$ and $u_i$ variables.

There exists a universal formal group law $F_\LL$, and its coefficient ring $\LL$ is called the \emph{Lazard ring}. This ring can be constructed as the quotient of the polynomial ring $\ZZ[A_{i,j}]_{i,j>0}$ by the relations imposed by the three axioms above. The images of the variables $A_{i,j}$ in the quotient ring are the coefficients $a_{i,j}$ of the formal group law $F_\LL$. It is shown in \cite{Levine-Morel} that $\Omega^*(\Spec k)$ is isomorphic as a graded ring to $\LL$ with grading induced by letting $A_{i,j}$ have degree $-i-j+1$. The power series $F_\LL(u,v)$ is then homogeneous of degree $1$ if $u$ and $v$ both have degree $1$. 

The formal group law on $\LL$ describes the first Chern class operators of tensor products of line bundles:
\[ \ch(L\otimes M) = F_\LL(\ch(L), \ch(M))\]
for any line bundles $L$ and $M$ on any scheme $X$ in $\Schk$.


\section{Operational Equivariant Cobordism of Toric Varieties}     \label{section.cobordism.toric}

Let $X_\Delta$ be a smooth quasiprojective toric variety corresponding to a fan $\Delta$. Krishna and Uma in \cite{Krishna-Uma} showed that the $T$-equivariant cobordism ring of $X_\Delta$ is isomorphic to the ring of piecewise power series on the fan $\Delta$. Our goal here is to show that for any fan $\Delta$, the ring of piecewise power series on $\Delta$ is isomorphic to the operational $T$-equivariant cobordism ring of $X_\Delta$. When $X_\Delta$ is smooth, this follows from the result of Krishna and Uma by \Po duality. For singular $X_\Delta$ we follow the argument of Payne \cite{Payne} in the case of Chow theory to reduce to the smooth case.

We use the standard notation for toric varieties \cite{Fulton}. Let $N \isom \ZZ^n$ be a lattice. It determines a split torus $T$ with character lattice $M= \Hom(N,\ZZ)$. A toric variety $X_\Delta$ is defined by a fan $\Delta$ in $N$. 

Every quasiprojective toric variety $X_\Delta$ with torus $T$ is in the category ${T}{\textnormal{-}}{\operatorname{Sch}_{k}}$, since each line bundle on such variety admits a $T$-linearization. 
We will write $\Omega_*^T(X_\Delta)$ for the $T$-equivariant cobordism group of $X_\Delta$. For a smooth $X_\Delta$  we also use the cohomological notation $\Omega^*_T(X_\Delta) = \Omega_{\dim X_\Delta-*}^T(X_\Delta)$. The operational $T$-equivariant cobordism ring is denoted $\Omega^*_T(X_\Delta)$. For smooth $X_\Delta$, the two definitions of  $\Omega^*_T(X_\Delta)$ are identified via \Po duality.

  
\subsection{Graded Power Series Rings}

Let $A=\oplus_{i\in\ZZ} A_i$ be a commutative graded ring. The \emph{graded power series ring} is
\[ A\llbracket t_1,t_2,\ldots,t_n\rrbracket_{gr} = \oplus_{d\in\ZZ} S_d.\]
 Here $S_d$ is the group of degree $d$ homogeneous power series  $\sum_I a_I t^I$, where the sum runs over multi-indices $I=(i_1,\ldots,i_n)\in \ZZ_{\geq 0}^n$, $t^I=  t_1^{i_1}\cdots t_n^{i_n}$, and $a_I\in A_{d-i_1-\cdots - i_n}$.

Let $T$ be a torus determined by a lattice $N$, and let $\chi_1,\ldots,\chi_n$ a basis for the dual lattice $M$. It is shown in \cite{Krishna} that the equivariant cobordism ring of a point with trivial $T$-action is isomorphic to
\[ \Omega^*_T(\Spec k) \isom \LL\llbracket t_1,\ldots,t_n\rrbracket_{gr},\]
with $t_i$ corresponding to the first Chern class transformation  $\ch^T(L_{\chi_i})$ of the equivariant line bundle $L_{\chi_i}$. More generally, for any smooth  $G$-variety $X$, let $T$ act trivially on $X$. Then
\begin{equation}\label{eq-triv-act} \Omega^*_{G\times T} (X) \isom \Omega^*_G(X)\llbracket t_1,\ldots,t_n\rrbracket_{gr},\end{equation}
where, as before, $t_i$ correspond to the first Chern class transformations of $L_{\chi_i}$ pulled back to $X$. Both of these formulas are proved using the projective bundle formula from \cite{Levine-Morel} in the cobordism ring.

As a matter of terminology, we identify $t_i$ with $\chi_i$ and view these variables as linear functions on $N$ (or on $N_\RR = N\otimes\RR$). The ring $S=  \LL\llbracket t_1,\ldots,t_n\rrbracket_{gr}$ is called the ring of graded power series on $N$ (or on $N_\RR$). The variables $t_i$ have the following functorial property. Let  $N'\to N$ be a morphism of lattices, giving rise to an algebraic group homomorphism $T'\to T$ of the corresponding tori. Assume that the image of $N'$ in $N$ is saturated. We can then choose a basis $\chi_1,\ldots,\chi_n$ for $M$ and $\nu_1,\ldots,\nu_m$ for $M' = \Hom(N',\ZZ)$ so that the dual map $M\to M'$ takes:
\[ \chi_i \mapsto \begin{cases} 
\nu_i & \text{ for $i=1,\ldots,l$,}\\
0 & \text{ for $i=l+1,\ldots,n$.}
\end{cases}
\]
Then the pull-back map 
\[ \Omega^*_T(\Spec k) \isom \LL\llbracket t_i,\ldots,t_n\rrbracket_{gr} \to \Omega^*_{T'}(\Spec k) \isom \LL\llbracket u_1,\ldots,u_m\rrbracket_{gr} \]
is an $\LL$-algebra homomorphism taking 
\[ t_i \mapsto \begin{cases} 
u_i & \text{ for $i=1,\ldots,l$,}\\
0 & \text{ for $i=l+1,\ldots,n$.}
\end{cases}
\]
(The map of variables is extended to the morphism of power series rings in an obvious way.) We view the morphism of power series rings as pulling back power series from $N$ to $N'$.

 As a special case, consider a  saturated sublattice $N'\subset N$ giving rise to a subtorus $T'\subset T$. When the bases for $M$ and $M'$ are chosen appropriately, the pull-back map 
\[ \Omega^*_T(\Spec k) \to \Omega^*_{T'}(\Spec k)\]
can then be identified with restriction of power series from $N$ to $N'$.

A little caution must be used when identifying $\Omega^*_T(\Spec k)$ with the ring of power series $\LL\llbracket t_1,\ldots, t_n\rrbracket_{gr}$. This identification depends on the chosen basis $\chi_1,\ldots,\chi_n$ for $M$. When we change the basis to $\nu_1,\ldots,\nu_n$, with the corresponding power series ring $\LL\llbracket u_1,\ldots,u_n\rrbracket_{gr}$, then the relation between $t_i$ and $u_j$ is not linear like the relation between $\chi_i$ and $\nu_j$. If $\nu_j = \sum_i a_{ji} \chi_i$, then one has to use the formal group law $F_\LL$ to express
\[ u_j = [a_{j1}]_{F_\LL} t_1 +_{F_\LL} [a_{j2}]_{F_\LL} t_2 +_{F_\LL} \cdots +_{F_\LL} [a_{jn}]_{F_\LL} t_n.\]
The result is a power series with only linear terms in $t_i$ equal to $\sum_i a_{ji}t_i$.

There is a one-to-one correspondence between cones $\sigma\in\Delta$ and $T$-orbits in $X_\Delta$. Let $O_\sigma$ be the orbit corresponding to a cone $\sigma$. Then the stabilizer (of a point) of $O_\sigma$ is the subtorus $T'$ of $T$ corresponding to the sublattice  $N' = \Span \sigma \cap N \subset N$. Morita isomorphism \cite{Krishna-Uma} then gives
\[ \Omega^*_T(O_\sigma) = \Omega^*_T(T\times^{T'} \Spec k) \isom \Omega^*_{T'}(\Spec k).\]
We denote  $S_\sigma = \Omega^*_T(O_\sigma) \isom \LL\llbracket t_1,\ldots,t_m\rrbracket_{gr}$,  $m=\dim\sigma$, and call it the ring of graded power series on $\sigma$. When $\tau$ is a face of $\sigma$, we have the inclusion of lattices $\Span\tau\cap N \subset \Span \sigma \cap N$, giving rise to the restriction map $S_\sigma\to S_\tau$. With the basis of $M$ chosen compatibly, this restriction map is the restriction of power series.


\subsection{The Sheaf of Piecewise Graded Power Series}

A \emph{sheaf} $F$ \emph{of abelian groups on a fan} $\Delta$ consists of the data:
\begin{enumerate}
 \item An abelian group $F_\sigma$ for every $\sigma\in\Delta$, called the \emph{stalk} of $F$ at $\sigma$.
\item A morphism $\res^\sigma_\tau: F_\sigma\to F_\tau$ for $\tau\leq \sigma$, called the \emph{restriction map}.
\end{enumerate}
The restriction maps must satisfy:
\begin{enumerate}
 \item $\res^\sigma_\sigma = \id_{F_\sigma}$.
\item $\res^\tau_\rho \circ \res^\sigma_\tau =  \res^\sigma_\rho$ for $\rho\leq\tau\leq\sigma$.
\end{enumerate}

One can give a topology on the set of cones $\Delta$ by defining the open sets to be subfans of $\Delta$. Then $F$ described above gives a sheaf on this topological space.

The group of global sections $F(\Delta)$ of a sheaf $F$ is the set of $(f_\sigma)\in \prod_{\sigma\in\Delta} F_\sigma$, such that $\res^\sigma_\tau(f_\sigma) = f_\tau$ for every $\tau\leq \sigma$. To give a global section of $F$, it suffices to give $f_\sigma$ for maximal cones $\sigma$ only.

For a cone $\rho\in\Delta$, the \emph{star} $\St \rho = \{ \sigma\in \Delta| \rho\leq \sigma\}$ is a closed set in the fan topology. Let $F( \St\rho)$ be the group of global sections of the sheaf $F|_{\St\rho}$. It consists of elements $(f_\sigma)\in \prod_{\sigma\in\St \rho} F_\sigma$, such that $\res^\sigma_\tau(f_\sigma) = f_\tau$ for every $\tau\leq \sigma \in \St \rho$.

Define the \emph{sheaf} $PPS$ \emph{of piecewise power graded series on} $\Delta$ \emph{with coefficients in the Lazard ring} as the sheaf with stalks $PPS_\sigma = S_\sigma$ and restriction maps $\res^\sigma_\tau$ the restriction maps $S_\sigma\to S_\tau$ as described above. The global sections $PPS(\Delta)$ are called \emph{piecewise graded power series with coefficients in the Lazard ring on} $\Delta$. For short, we refer to $PPS$ as the sheaf of piecewise power series on $\Delta$ and to its global sections $PPS(\Delta)$ as piecewise power series on $\Delta$. 
Note that the sheaf $PPS$ is a sheaf of graded $S$-algebras: all stalks $S_\sigma$ are graded $S$-algebras and the restriction maps are homomorphisms of graded $S$-algebras. It follows that $PPS(\Delta)$ is also a graded $S$-algebra. Similarly, $PPS(\St \rho)$ is the graded $S$-algebra of piecewise power series on $\St \rho$.

We must again use caution when working with the restriction maps $S_\sigma\to S_\tau$. These are restrictions of power series, once we have chosen a suitable basis for $M$. When $\tau'$ is another face of $\sigma$, representing the restriction map $S_\sigma\to S_{\tau'}$ as restriction of power series may require a different basis for $M$ that is compatible with $\sigma$ and $\tau'$. The change of basis gives rise to an isomorphism of the graded power series rings that involves the formal group law.


\subsection{Operational Equivariant Cobordism of Toric Varieties}             \label{operational.cobordism.TV}

The inclusion map $i_\sigma: O_\sigma \hookrightarrow X_\Delta$ is an l.c.i. morphism when $X_\Delta$ is smooth, hence there exists the pull-back map:
\[ i_\sigma^*: \Omega^*_T(X_\Delta) \to \Omega^*_T(O_\sigma) = S_\sigma.\]
The following theorem was proved in \cite{Krishna-Uma}.

\begin{theorem} (Krishna-Uma) \label{thm-Krishna-Uma} Let $X_\Delta$ be a smooth quasiprojective toric variety. Then the morphism of $S$-algebras
 \[  \Omega^*_T(X_\Delta) \xrightarrow{(i_\sigma^*)} \prod_{\sigma\in\Delta} S_\sigma \]
is injective and the image is equal to the $S$-algebra $PPS(\Delta)$ of piecewise graded power series on $\Delta$ with coefficients in the Lazard ring.
\end{theorem}

In the proof of Krishna and Uma, the group $\Omega^*_T(X_\Delta)$ stands for the cohomological notation of  $\Omega_{\dim X_\Delta-*}^T(X_\Delta)$ and the maps $i_\sigma^*$ are l.c.i. pull-backs. We claim that the same statement is true for general $X_\Delta$ when $\Omega^*_T(X_\Delta)$ stands for the operational cobordism group and $i_\sigma^*$ is the pull-back morphism in the operational theory.
 
\begin{theorem} \label{thm.toric} Let $X_\Delta$ be a quasiprojective toric variety. Then the morphism of $S$-algebras
 \[  \Omega^*_T(X_\Delta) \xrightarrow{(i_\sigma^*)} \prod_{\sigma\in\Delta} S_\sigma \]
is injective and the image is equal to the $S$-algebra $PPS(\Delta)$ of piecewise graded power series on $\Delta$ with coefficients in the Lazard ring.
\end{theorem}

\begin{proof}
 
We prove the theorem by induction on $\dim X_\Delta$.
In the inductive proof we will need a slightly stronger statement. Let $T'$ be another split torus with character lattice $M'$. Fix a basis $\nu_1,\ldots,\nu_m$ of $M'$ and the corresponding isomorphism $\Omega^*_{T'}(\Spec k) \isom \LL\llbracket u_1,\ldots,u_m\rrbracket_{gr}$. We let $T'$ act trivially on $X_\Delta$. 

Instead of the theorem, we prove the following stronger statement:
\begin{claim} The morphism 
 \[  \Omega^*_{T\times T'} (X_\Delta) \xrightarrow{(i_\sigma^*)} \prod_{\sigma\in\Delta} S_\sigma\llbracket u_1,\ldots,u_m\rrbracket_{gr}\]
is injective and the image is equal to the graded $S\llbracket u_1,\ldots, u_m\rrbracket_{gr}$-algebra $PPS(\Delta)\llbracket u_1,\ldots, u_m\rrbracket_{gr}$.
\end{claim}

To simplify notation, let $\tilde{T} = T\times T'$, $\tilde{S} = S\llbracket u_1,\ldots,u_m\rrbracket_{gr}$, $\tilde{S}_\sigma = S_\sigma\llbracket u_1,\ldots,u_m\rrbracket_{gr}$. The restriction maps $\tilde{S}_\sigma \to \tilde{S}_\tau$ are obtained by applying  the restriction maps $S_\sigma\to S_\tau$ to the coefficients of a power series. This collection of rings and restriction maps defines a sheaf $\widetilde{PPS}$ on $\Delta$, with global sections $\widetilde{PPS}(\Delta) = PPS(\Delta)\llbracket u_1,\ldots,u_m\rrbracket_{gr}$. We also call the $\tilde{S}$-algebra   $\widetilde{PPS}(\Delta)$ the ring of piecewise power series on $\Delta$.

For smooth $X_\Delta$ the claim follows from Theorem~\ref{thm-Krishna-Uma} and Equation~(\ref{eq-triv-act}).

When $X_\Delta$ is singular, we can resolve its singularities by a sequence of star subdivisions of $\Delta$:
\[ X_\Delta \lar X_{\Delta'} \lar \cdots\lar X_{\Delta''}.\]
We may assume by induction on the number of star subdivisions that the claim holds for $X_{\Delta'}$. The morphism $f:X_{\Delta'}\to X_{\Delta}$ is the blowup of $X_\Delta$ along a closed subscheme $C\subset X_\Delta$ with support $|C|$ equal to the orbit closure  $V_\pi =  \overline{O}_\pi$, where $\pi\in\Delta$ is the cone containing the subdivision ray in its relative interior. Let $\rho\in\Delta'$ be the new ray. Then the exceptional divisor $E=f^{-1}(C)$ has support $|E| = V_\rho$. The morphism $f$ is a birational envelope. In order to use Theorem~\ref{thm-ind}, we need to identify $\Omega^*_{\tilde{T}}(C), \Omega^*_{\tilde{T}}(E)$ and the pull-back map between them. 

\begin{lemma} For any $0\neq \pi\in\Delta$, the map 
\[  \Omega^*_{\tilde{T}}(V_\pi) \xrightarrow{(i_\sigma^*)} \prod_{\sigma\in\St(\pi)} \tilde{S}_\sigma\]
is injective and the image is equal to the graded $\tilde{S}$-algebra $\widetilde{PPS}(\St \pi)$ of piecewise power series on $\St \pi$. 
\end{lemma}
\begin{proof}
 The orbit closure $V_\pi$ is again a toric variety corresponding to the fan $\Delta_\pi$ that is the image of  $\St \pi$ in $N/(\Span\pi\cap N)$. We split $N=N_1\oplus N_2$, where $N_2 = \Span\pi\cap N$, and consider $\Delta_\pi$ as a fan in $N_1$. Splitting the lattice corresponds to decomposing the torus $T=T_1\times T_2$, where $T_1$ is the big torus in $V_\pi$ and $T_2$ acts trivially on $V_\pi$. Since $\dim V_\pi < \dim X_\Delta$, we have by induction
\begin{equation} \label{eq11}  
\Omega^*_{T_1\times T_2 \times T'}(V_\pi) \isom PPS(\Delta_\pi) \llbracket t_1,\ldots,t_l\rrbracket_{gr}\llbracket u_1,\ldots,u_m\rrbracket_{gr}, 
\end{equation}
where $t_1,\ldots, t_l$ correspond to a basis of the dual lattice of $N_2$, hence $S_\pi \isom \LL\llbracket t_1,\ldots,t_l\rrbracket_{gr}$. The isomorphism in (\ref{eq11}) is defined by pull-back maps 
\[ i_\sigma^*: \Omega^*_{T_1\times T_2 \times T'}(V_\pi) \to \Omega^*_{T_1\times T_2 \times T'}(O_\sigma)\]
for $\sigma \in \St \pi$. Here we used the fact that if $\bar{\sigma}$ is the image of $\sigma$ in $\Delta_\pi$, then $O_\sigma = O_{\bar{\sigma}}$ as subsets of $X_\Delta$. Since $S_\sigma = S_{\bar{\sigma}}\llbracket t_1,\ldots,t_l\rrbracket_{gr}$ for every $\sigma\in\St \pi$, it follows that   
\[  PPS(\Delta_\pi) \llbracket t_1,\ldots,t_l\rrbracket_{gr}  \isom PPS(\St \pi). \]
This implies the statement of the lemma.
\end{proof}

We can apply the previous lemma to $\rho\in\Delta'$, to get that $\Omega^*_{\tilde{T}}(V_\rho)$ is isomorphic to $\widetilde{PPS}(\St \rho)$. Moreover, the pull-back map $\Omega^*_{\tilde{T}}(X_{\Delta'}) \to \Omega^*_{\tilde{T}}(V_\rho)$ is simply the restriction of piecewise power series $\widetilde{PPS}(\Delta') \to \widetilde{PPS}(\St \rho)$.

Next we describe the pull-back morphism $\Omega^*_{\tilde{T}}(V_\pi) \to \Omega^*_{\tilde{T}}(V_\rho)$.
Note that every cone in $\sigma\in \St\rho$ lies in some cone of $\tau\in \St\pi$. Choosing a basis for $M$ compatibly with $\sigma$ and $\tau$ gives a restriction map of power series $\tilde{S}_\tau\to \tilde{S}_\sigma$.  It is easy to see that these maps combine to give a  well defined pull-back map of piecewise power series $\widetilde{PPS}(\St\pi)\to \widetilde{PPS}(\St\rho)$.

\begin{lemma}
The pull-back morphism 
\[\Omega^*_{\tilde{T}}(V_\pi)\isom \widetilde{PPS}(\St\pi)  \to \Omega^*_{\tilde{T}}(V_\rho) \isom \widetilde{PPS}(\St\rho)\]
is the pull-back of piecewise power series.
\end{lemma}

\begin{proof}
Define a map $\phi: \Delta'\to \Delta$ so that $\phi(\sigma)$ is the smallest cone in $\Delta$ containing $\sigma$. Then the map $f:X_{\Delta'}\to X_\Delta$ takes $O_\sigma$ onto $O_{\phi(\sigma)}$. Choosing an appropriate basis for $M$, the pull-back morphism
\[ \Omega^*_{\tilde{T}}(O_{\phi(\sigma)} ) = \tilde{S}_{\phi(\sigma)} \to \Omega^*_{\tilde{T}}(O_\sigma) = \tilde{S}_\sigma \]
is the restriction of power series. 

Consider the diagram
\[ \begin{CD}  
 \Omega^*_{\tilde{T}}(V_\rho) @>>> \prod_{\sigma\in\St\rho} \Omega^*_{\tilde{T}}(O_\sigma)  \\
@AAA @AAA\\
 \Omega^*_{\tilde{T}}(V_\pi) @>>> \prod_{\tau\in\St\pi} \Omega^*_{\tilde{T}}(O_\tau),\\
\end{CD} \]
where all maps are pull-back morphisms. The right vertical map sends $(s_\tau)$ to $(t_\sigma)$, such that $t_\sigma$ is the image of $s_{\phi(\sigma)}$; this map restricts to the pull-back of piecewise power series $\widetilde{PPS}(\St\pi)\to \widetilde{PPS}(\St\rho)$.
Now it suffices to show that the diagram commutes, which follows from the commutativity of the diagram:
\[ \begin{CD}  
 V_\rho @<<< O_\sigma  \\
@VVV @VVV\\
 V_\pi @<<< O_{\phi(\sigma)}\\
\end{CD} \]
for every $\sigma\in\St \rho$.
\end{proof}

To finish the proof of the claim, we apply Theorem~\ref{thm-ind} with $S_1 = C^{red} = V_\pi$ and $E_1 = E^{red} = V_\rho$. Since  the pull-back map $\Omega^*_{\tilde{T}}(V_\pi)\to \Omega^*_{\tilde{T}}(V_\rho)$ is injective, Theorem~\ref{thm-ind} implies that we have a Cartesian diagram, with all maps pull-backs:
\[ \begin{CD}  
 \Omega^*_{\tilde{T}}(X_\Delta) @>>>  \Omega^*_{\tilde{T}}(V_\pi)  \\
@VVV @VVV\\
 \Omega^*_{\tilde{T}}(X_{\Delta'}) @>>>  \Omega^*_{\tilde{T}}(V_\rho).\\
\end{CD} \]
The following diagram of piecewise power series and pull-back maps is clearly Cartesian:
\[ \begin{CD}  
 \widetilde{PPS}(\Delta) @>>>  \widetilde{PPS}(\St \pi)  \\
@VVV @VVV\\
 \widetilde{PPS}(\Delta') @>>>  \widetilde{PPS}(\St \rho).\\
\end{CD} \]
This implies that $\Omega^*_{\tilde{T}}(X_\Delta)$ is isomorphic to $\widetilde{PPS}(\Delta)$. To see that the isomorphism comes from the maps $(i_\sigma^*)$ as stated in the claim, we only need to consider the commutative diagram:
\[ \begin{CD}  
 \Omega^*_{\tilde{T}}(X_\Delta) @>>> \prod_{\sigma\in\Delta} S_\sigma  \\
@VVV @VVV\\
 \Omega^*_{\tilde{T}}(X_{\Delta'}) @>>>  \prod_{\tau\in\Delta'} S_\tau,\\
\end{CD} \]
where the right vertical morphism is defined by the map $\phi: \Delta'\to\Delta$  as above. Commutativity of the diagram shows that the inclusion map 
\[ \Omega^*_{\tilde{T}}(X_\Delta)  \hookrightarrow \Omega^*_{\tilde{T}}(X_{\Delta'})  \isomto \widetilde{PPS}(\Delta') \]
factors through $\prod_{\sigma\in\Delta} S_\sigma$. The image of $\Omega^*_{\tilde{T}}(X_\Delta)$ in $\prod_{\sigma\in\Delta} S_\sigma$ is then equal to $\widetilde{PPS}(\Delta)$.
\end{proof}


\bibliographystyle{plain}
\bibliography{cobordismbib}

\end{document}